\RequirePackage{fix-cm}
\documentclass[smallcondensed]{svjour3}     
\pdfoutput=1

\smartqed                  
\usepackage{graphicx}

\usepackage{mathptmx}      

\usepackage[utf8]{inputenc}

\usepackage[english]{babel}

\usepackage{amssymb, amsmath}

\usepackage{hyperref}

\newcommand\E{\mathcal{E}}
\newcommand{\dt}{\Delta t}
\newcommand\Z{\mathbb{Z}}
\newcommand\dx{\Delta x}

\newcommand\pd[2]{\frac{\partial #1}{\partial #2}}

\newcommand\BigO{\mathcal{O}}
\newcommand\R{\mathbb{R}}

\newcommand{\Tm}{\mathcal{T}}
\newcommand\de{\delta}
\newcommand\eps{\epsilon}
\renewcommand\L{\mathcal{L}}
\newcommand\h{\bar{h}}
\newcommand\x{{\bf x}}

\newtheorem{rmk}{Remark}
\newtheorem{thm}{Theorem}
\newtheorem{defn}{Definition}

\journalname{Journal of Scientific Computing}

\begin{document}

\title{
High-order multiderivative time integrators for hyperbolic conservation laws
}

\author{David C. Seal \and
        Yaman Güçlü   \and \\
        Andrew J. Christlieb}
\institute{
    David C. Seal \at
    Department of Mathematics \\
    Michigan State University \\
    619 Red Cedar Road \\
    East Lansing, MI 48824, USA \\
    Tel.: +1(517) 884-1456 \\
    Fax.: +1(517) 432-1562 \\
    \email{seal@math.msu.edu} 
\and
    Yaman Güçlü \at
    Department of Mathematics \\
    Michigan State University \\
    East Lansing, MI 48824, USA
\and
    Andrew J. Christlieb \at
    Department of Mathematics and
    Department of Electrical and Computer Engineering \\
    Michigan State University \\
    East Lansing, MI 48824, USA \\
}

\date{Received: date / Accepted: date}

\maketitle

\begin{abstract}
Multiderivative time integrators have a long history of development for 
ordinary differential equations, and yet to date, only a small subset of 
these methods have been 
explored as a tool for solving 
partial differential 
equations (PDEs).
This large class of time integrators include all popular (multistage) 
Runge-Kutta as well as single-step (multiderivative) Taylor methods.
(The latter are commonly referred to as Lax-Wendroff methods when applied to 
PDEs.)
In this work, we offer explicit multistage multiderivative time integrators 
for hyperbolic conservation laws. 
Like Lax-Wendroff methods, multiderivative integrators permit the 
evaluation of higher derivatives of the unknown in order to decrease the 
memory footprint and communication overhead. 
Like traditional Runge-Kutta methods, multiderivative integrators admit 
the addition of extra stages, 
which introduce extra degrees of freedom that can be 
used to increase the order of accuracy or modify the region of absolute 
stability.
We describe a general framework for how these methods can be applied to two 
separate spatial discretizations: the discontinuous Galerkin (DG) method and 
the finite difference essentially non-oscillatory (FD-WENO) method.
The two proposed implementations are substantially different: for DG we 
leverage techniques that are closely related to generalized Riemann solvers; 
for FD-WENO we construct higher spatial derivatives with
central differences.
Among multiderivative time integrators, we argue that multistage 
two-derivative methods have the greatest potential for multidimensional 
applications, because they only require 
the flux function and its Jacobian, which is readily available.
Numerical results
indicate that multiderivative methods are indeed competitive with popular strong stability
preserving time integrators.


\end{abstract}




\section{Introduction}
\label{sec:introduction}

In this work we revisit classical ordinary differential equation (ODE) solvers
known as multiderivative (Obreshkoff \cite{Obreschkoff40}) methods.  
It will be shown that this large class of time integrators
include all explicit Runge-Kutta and Taylor methods\footnote{When applied to
partial differential equations, Taylor methods are commonly referred to as 
Lax-Wendroff methods.} as special cases.
In particular, 
we demonstrate how a multiderivative ODE method can be used to solve
hyperbolic conservation laws.
We begin by presenting the definitions and notation used throughout 
this work.

A \emph{conservation law} is a partial differential equation (PDE)
defined by a flux function $R$ of the form
\begin{equation}\label{eqn:hyperbolic-law}
   q_{,t} + \nabla_\x \cdot R(q) = 0, \quad q(0,{\bf x}) = q_0( \x ), \quad \x \in \Omega \subseteq \R^d,
\end{equation}
where the solution $q(t, {\bf x}): \R^+ \times \R^d \to \R^m$ is a
vector of $m$ conserved quantities.
In dimension $d$, this initial value problem is defined by $m$ equations
with prescribed initial conditions $q_0 : \R^d \to \R^m$.
We denote the \emph{flux function} $R: \R^m \to \R^d \times \R^m$
with $R = [f^{(1)}, f^{(2)}, \dots , f^{(d)}]^T$.
We say \eqref{eqn:hyperbolic-law} is \emph{hyperbolic} if the matrix
\begin{equation}
   {\bf n}^{(1)} \pd{f^{(1)}}{q}( q ) + {\bf n}^{(2)}  \pd{f^{(2)}}{q}( q ) + \cdots {\bf n}^{(d)}  \pd{f^{(d)}}{q}( q )
\end{equation}
is diagonalizable for every  ${\bf n} \in \R^d$ satisfying
$\| {\bf n} \| = 1$ and $q$ in the domain of interest.

Numerical methods for solving \eqref{eqn:hyperbolic-law} require a discretization 
technique for space as well as a (possibly coupled) 
discretization technique for time.
The vast majority of time stepping discretizations fall into one of
two distinct categories:
%
\begin{itemize}
   \item {\bf Method of lines formulation}.

A method of lines (MOL) solver for  \eqref{eqn:hyperbolic-law}
separates the spatial discretization from the time evolution.
Starting with $q_{,t} = - \nabla_{\x} \cdot R(q)$,
one defines a spatial discretization
$q^h$ of the continuous 
variable $q$, which could be tracking point values (finite difference, spectral
methods), cell averages (finite volume methods), 
or coefficients of basis functions (finite element methods).
This operation defines a function
$\L(q^h) = - \nabla_{\x} \cdot R(q^h)$,
that in turn defines a large ODE system of the form: 
\begin{align}
   q^h_{,t} = \L(q^h).
\end{align}
Once this discretization has been parsed,
one may apply an appropriate ODE integrator to this problem:
for hyperbolic conservation laws, 
explicit time-stepping methods are usually preferred.

\item {\bf Lax-Wendroff (Taylor) formulation}.

The Lax-Wendroff \cite{LxW60} procedure is a numerical scheme 
that updates the solution
using finitely many terms from the Taylor series of the function.
Here,
one first expands $q(t, \x)$ in time about $t=t^n$:
\begin{equation} \label{eqn:lax-wendroff}
   q(t,\x) = q^n + (t-t^n) q^n_{,t} + \frac{(t-t^n)^2}{2!} q^n_{,tt}
   + \cdots,
\end{equation}
and then each time derivative is replaced with a spatial derivative
via the Cauchy-Kowalewski procedure:
\begin{equation}\label{eqn:cauchy-kowalewski}
\begin{aligned}
   q_{,t}  &= -\nabla_\x \cdot R(q), \\
   q_{,tt} &= -\nabla_{\x} \cdot R(q)_{t} = - \nabla_{\x} \cdot \left( R'(q) q_{,t} \right) = 
   \nabla_{\x} \cdot \left( R'(q) \cdot \nabla_{\x} \cdot R(q) \right), \\
   \vdots &
\end{aligned}
\end{equation}
%
Further derivatives are required for higher order variants, and of course, one
still needs to choose a spatial discretization.
Inserting $t=t^n+\dt$
produces a single-stage, single-step method.  
In addition, this is
fundamentally different than the MOL discretization, because the physical and
temporal variables are intimately intertwined through the choice of the spatial
discretization of the right hand side of \eqref{eqn:lax-wendroff}.
\end{itemize}

Multiderivative Runge-Kutta time integrators form the bridge that unifies these
two disparate families of methods by defining a framework that includes
each of them as special cases.
We will see that the generalization presented in this work
makes use of techniques
used in the development of high-order method of 
lines formulations as well as high-order Lax-Wendroff type time discretizations.

\subsection{High-order method of lines formulation for PDEs}

The most popular high-order time integrators for hyperbolic problems fall 
into the method of lines category.  By and large, the predominant viewpoint from
the community is to develop spatial discretizations separate from time 
integrators.  This is 
incredibly attractive from a software engineering 
perspective:
one can envision developing a code that completely decouples the ODE
technology from the spatial discretization.
In addition, the MOL formulation invites developers to concentrate efforts
on ODE integrators as a separate entity from the PDE.
However, this idealization is lacking given that a numerical scheme
is intended to solve a PDE, and therefore one needs to respect the choice of
spatial discretization not only when selecting an ODE integrator, but also
when developing one.  In contrast, Taylor methods require a recognition of the 
particular choice of spatial discretization.

\subsection{High-order Lax-Wendroff methods for PDEs}

The Lax-Wendroff procedure is much older than either Lax or Wendroff.
Indeed, a more appropriate name would be the Cauchy-Kowalewski procedure,
where Cauchy and Kowalewski sought methods that
could aid them in proving existence and uniqueness for solutions to PDEs.
Their combined method, known as the Cauchy-Kowalewski procedure, was
presented in
equation \eqref{eqn:cauchy-kowalewski}, and is derived from 
Brook Taylor's method,
who invented equation \eqref{eqn:lax-wendroff} in the 1700's.
For a modern (mid-$20^\text{th}$ century) proof and review of the 
Cauchy-Kowalewski procedure see Friedman \cite{Friedman61} or 
Fusaro \cite{Fusaro68} and references therein.
In 1960, Peter Lax and Burton Wendroff \cite{LxW60} realized
the Cauchy-Kowalewski procedure could be used as a numerical method.
They started with the theoretical groundwork developed by Cauchy
and Kowalewski and derived a  
numerical scheme for solving PDEs.
Therefore, this entire procedure is often cited as the Lax-Wendroff
method within the numerical analysis community.

The original Lax-Wendroff method was a second-order numerical
discretization of the Cauchy-Kowalewski procedure, and
over the past decade, much work has been put forth to define
high-order variants of this method.
In 2002, Toro and Titarev started work on a series of papers
that became the basis for the so-called ADER (Arbitrary DERivative) methods
that define high-order versions of the Lax-Wendroff procedure 
\cite{proceedings:TitarevToro02,ToroTitarev02,ToroTitarev05:JCP,TitarevToro05:JCP:systems,ToroTitarev05:JSP}.
During that same time period, Daru and Tenaud \cite{DaruTenaud04}
explored high-order monotonicity preserving single-step methods, and they 
derived TVD flux limiters to control spurious oscillations.
In 2003, Jianxian Qiu and his collaborators demonstrated a high-order extension
of the Lax-Wendroff procedure using finite difference weighted essentially non-oscillatory (WENO) methods 
\cite{QiuShu03}.  
Later on, they applied the same procedure to Hamilton-Jacobi systems 
as well as the shallow water 
equations \cite{Qiu07,LuQiu11}.
Additionally, Qiu, Dumbser and Shu showed how to apply the Lax-Wendroff 
scheme to the discontinuous Galerkin (DG) method 
\cite{QiuDumbserShu05}, and shortly thereafter, 
Dumbser and Munz 
followed a similar procedure for constructing DG methods to arbitrarily high
orders of accuracy using generalized Riemann solvers
\cite{DuMu06}.
High-order versions of a Lax-Wendroff discontinuous Galerkin method
have been investigated for ideal magnetohydrodynamic equations 
\cite{TaDuBaDiMu07}, and
explorations into various numerical flux functions
for the Lax-Wendroff DG method has also been carried
out \cite{JQiu07}.
It has already been noted that high-order schemes with Lax-Wendroff type
time discretizations can be implemented to carry a low-memory footprint 
\cite{LiuChengShu09}.

Much of the difficulty when constructing high-order versions of the 
Lax-Wendroff scheme comes from the necessity of defining higher spatial 
derivatives of the solution.
After producing the Jacobian of the flux function, the next time derivative
produces the Hessian 
of the flux function. Further derivatives require tensors
which grow vastly in size, and therefore,
one of the primary concerns with a high-order
Lax-Wendroff method is the burden of implementing higher
derivatives, especially in higher dimensions.

\subsection{High-order multistage multiderivative methods for PDEs}

In this work, we advocate the use of multistage multiderivative
integrators for solving hyperbolic conservation laws.  These time integrators
are the natural generalization of MOL formulations as well as pure Taylor 
(Lax-Wendroff) methods.  The introduction of higher derivatives allows one to design compact
stencils, and the introduction of degrees of freedom to Taylor methods
allows one to explore closely related alternatives.
We argue that the benefits of exploring multiderivative time integrators
include, but are not limited to the following:

\begin{itemize}
\item {\bf High-order accuracy} [{\it $3^{rd}$-order or higher}].  
Explicit multiderivative schemes can be constructed to arbitrarily high
orders of accuracy.
We focus on a single fourth-order method as our demonstrative
example.

\smallskip
\item {\bf Portability.}  
Access to the eigen-decomposition of a hyperbolic problem is a necessity.
Therefore, multistage multiderivative integrators that stop at the second derivative 
do not require anything above and beyond anything that is already called for,
and therefore, they are more portable than pure Lax-Wendroff (Taylor) methods.

\smallskip

\item {\bf Low-storage.}  
Multiderivative integrators carry a small memory
footprint.
By design, these integrators exchange storage for
extra FLOPs in order to attain high-order accuracy.
This feature is a desirable trait for high performance computing given than the
current trend is towards inexpensive FLOPs and expensive memory. 

%
%
\end{itemize}

The primary purpose of the present work is to demonstrate 
how multistage multiderivative integrators can be used
to solve PDEs.
Given the plethora of multiderivative methods from the ODE 
community, we choose to select demonstrative examples
that can be easily modified to accommodate all explicit 
multiderivative methods.  In particular, most of our numerical
results will focus on simulations for a particular 
fourth-order example that serves as a representative example of 
a method that falls outside the confines of the Runge-Kutta and Taylor families.

The outline of this paper is as follows:
we begin in \S\ref{sec:methods}
with a historical review of multiderivative integrators.
In \S\ref{sec:weno}, we describe the finite difference WENO scheme, and we
demonstrate how multiderivative technology can be applied to the WENO framework.
In \S\ref{sec:dg}, we continue by looking at multiderivative integrators for 
the discontinuous Galerkin method.
In \S\ref{sec:numerical_examples} we present numerical results for our
numerous numerical test problems, and in 
\S\ref{sec:conclusions}
we draw up conclusions and point to future work.


\section{High-order explicit multiderivative ODE integrators}
\label{sec:methods}

Multistage multiderivative integrators for PDEs require a blend of 
both the the method of lines (MOL) formulation
as well as the Lax-Wendroff formulation of 
\eqref{eqn:hyperbolic-law},
and in addition, one needs to select a method for the 
spatial discretization.
We begin our description of multistage multiderivative PDE solvers
with a historical overview of these methods within the confines of ODEs.
In particular, we focus on explicit multiderivative Runge-Kutta integrators,
which include single-derivative methods 
(e.g.\ classical Runge-Kutta) as well as
single-stage methods (e.g.\ Taylor) as special cases.
We begin with a review of multiderivative methods
in \S\ref{subsec:ode-integrators-review},
and then continue in 
\S\ref{subsec:ode-integrators-definitions} by defining a large class
of explicit multiderivative Runge-Kutta methods.
In \S\ref{subsec:ode-integrators-example}, we present model examples
of methods from this class.


\subsection{Multiderivative methods: a review}
\label{subsec:ode-integrators-review}

Numerical methods using multiderivative technology have a long history 
dating back to at least the early 1940's, and
some of the pioneering work for explicit schemes share a common ancestry
with implicit schemes.
In 1963, Stroud and Stancu \cite{StSt63} applied the quadrature method of
Tur\'an \cite{Tu50}, which generated an implicit, high-order multiderivative 
ODE solver.
Prior to Tur\'an's work, in 1940, Obreshkoff \cite{Obreschkoff40} derived discrete
quadrature formulae for integrating functions,
and much like Tur\'an did, Obreshkoff used extra derivatives of the function
for his quadrature rules.
When extra derivatives are included, one can obtain methods with
excellent properties for the numerical integration of ODEs, including
high-order accuracy involving fewer quadrature points than would otherwise
be required.
In 1972, Kastlunger and Wanner 
\cite{KaWa72} used the theory of Butcher trees to show
that Tur\'an's quadrature method
could be written as an implicit multiderivative Runge-Kutta scheme.
The following year, Hairer and Wanner \cite{HaWa73} defined
``multistep multistage multiderivative methods'', 
that to date, has stood the test of time as being a
broad categorical definition of numerical methods for solving ODEs.
A concise taxonomy of this large class of methods is 
presented in Figure \ref{fig:mmmm}.
In particular, their definition contains 
all Runge-Kutta and all linear multistep methods as well as additional
combinations, including so-called general linear methods 
(c.f.\ John Butcher's extensive review papers
\cite{Butcher85,Butcher96,Butcher06} 
for a description of general linear methods).
The textbooks of Hairer, N{\o}rsett, and Wanner  \cite{HaWa09,HaWa06}
contain excellent references. 

\begin{figure}
\begin{center}
\includegraphics[width=0.7\textwidth]{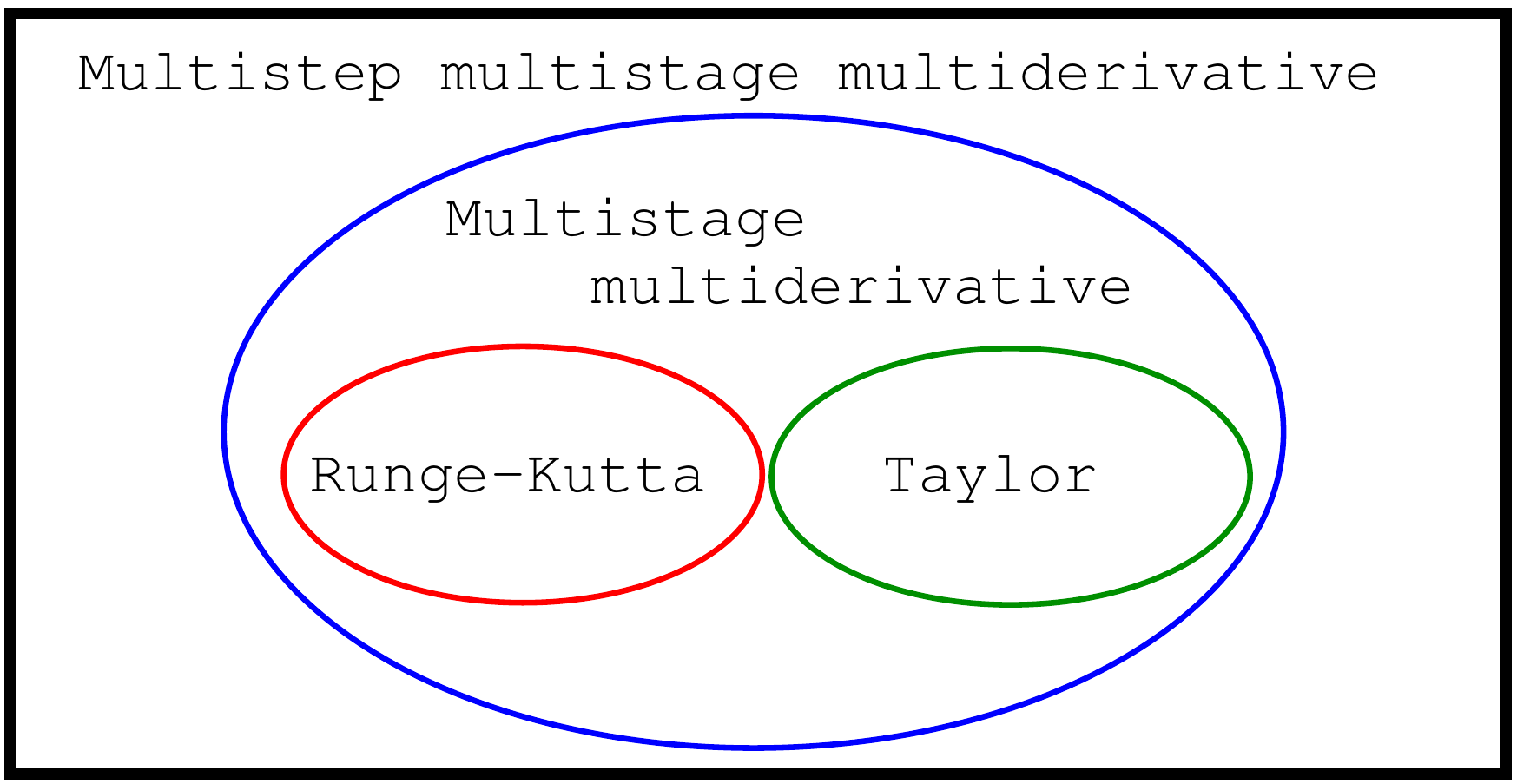}
\caption{A simple taxonomy of ODE solvers.
Multistep multistage multiderivative methods 
as defined by Hairer and Wanner \cite{HaWa73}
are the most inclusive class presented in this diagram.
Our focus is on multistage multiderivative methods
that include 
Runge-Kutta (a.k.a. multistage) and Taylor (a.k.a. multiderivative) 
as special cases.
\label{fig:mmmm}
}
\end{center}
\end{figure}

Our current focus is on explicit versions of multiderivative Runge-Kutta 
schemes, which needless to say, also have a long history of development.
Despite their age, these methods have 
seen little to no attention outside the ODE community, yet given the direction
of modern computer architecture, many of these methods may see use for solving
PDEs in the near future.
In 1952, Rudolf Zurm{\"u}hl 
\cite{Zurm52} 
investigated multiderivative Runge-Kutta
integrators, and later on, Erwin Fehlberg 
\cite{Fehlberg60,Fehlberg64}
derived an explicit multiderivative Runge-Kutta methods.
Early versions of Fehlberg's method
applied a single-derivative Runge-Kutta method to the modified variable that
is constructed by subtracting out $m$-derivatives of the original variable.
A decade later, Kastlunger and Wanner \cite{KaWa72-RK} 
extended Butcher's method to multiderivative Runge-Kutta methods.
In their work, they defined the order conditions for the coefficients 
in a multiderivative Butcher tableau, and in addition,
they showed that Fehlberg's method \cite{Fehlberg60,Fehlberg64}
can be written as a multiderivative Runge-Kutta process
with $m$-derivatives taken at a single node.
During that same decade, Shintani \cite{Shintani71,Shintani72}
worked on multiderivative Runge-Kutta methods.
Also in the 1970's, Bettis and Horn \cite{BettisHorn76} revisited
Fehlberg's scheme and reformulated it as an embedded Taylor method: 
for their celestial mechanics problem, 
they describe how the necessary Taylor series coefficients can be generated 
with little to no additional cost.
A decade later and unaware of the full history of the methods,
Mutsui \cite{Mitsui82} also worked on Runge-Kutta methods 
that leveraged extra information with extra derivatives.

The most recent work on explicit multiderivative integrators appears
to focus on redefining order conditions and demonstrating examples of 
methods from this class, much of which has been carried out independently
from previous work.
In 1986, Gekeler and Widmann \cite{GeWi86} used the theory of
Butcher trees to define the correct order conditions for 
multiderivative Runge-Kutta methods.  In their work,
they presented families of methods with orders ranging between four and seven.
Goeken and Johnson \cite{GoJo99,GoJo00} were unaware of the long standing
history of explicit multiderivative methods, and they derived their own
versions of explicit methods that are sub-optimal.
Within the past decade,
Yoshida and Ono \cite{OnYo03,OnYo04} and Chan and Tsai \cite{ChTs10}
used the theory of Butcher trees to define order conditions and presented
numerous examples.
The primary difference between the latter two works is the following:
Chan and Tsai used multiple stages for their methods, and they 
restrict their 
attention to using two derivatives of the unknown function;
Yoshida and Ono restrict their attention to two-stage methods,
and they admit arbitrarily many derivatives of the unknown function to be
evaluated at every quadrature point.
In other very recent work,
Nguyen-Ba, Bo{\v{z}}i{\'c}, Kengne and Vaillancourt \cite{BaBo09}
derived a nine-stage
explicit multiderivative Runge-Kutta
scheme that makes use of extra derivatives at a single quadrature point only,
much like the schemes Fehlberg
originally investigated \cite{Fehlberg60,Fehlberg64}.  In doing so,
the order conditions become simpler to navigate.

For the purposes of solving hyperbolic conservation laws, we view
using at most two-derivatives of the function as the most appropriate
choice given the opportunity to retain portable code.
Hyperbolic problems require a definition of the Jacobian of 
the flux function, which is precisely the term required to define
a two-derivative scheme.
Investigations into methods using extra derivatives would 
make for interesting future research.

\subsection{Multistage multiderivative methods: some definitions}
\label{subsec:ode-integrators-definitions}

Consider a system of ODEs, 
defined by
\begin{equation}\label{eqn:ode}
   \dot{y} = L(y), \quad y(0) = y_0, \quad t > 0.
\end{equation}
We use the letter $L$ in place of $f$ to avoid conflict with the flux
function defined later on in equation \eqref{eqn:1dhyper}.
Without loss of generality, we assume the system is autonomous.
Multiderivative methods make use of extra derivatives of \eqref{eqn:ode}.
If we take a single time derivative of \eqref{eqn:ode}, we see that
\begin{align}
   \ddot{y} = \dot{L} = L'(y)\, \dot{y} = \pd{L}{y} L( y ),
\end{align}
where $\pd{L}{y}$ denotes the partial derivative of $L$ with respect to $y$.

Higher derivatives can be computed recursively.  Define the $m^{th}$ 
derivative of $y$ as 
$y^{(m)} := \frac{d^m y}{dt^m}$, and observe that
$y^{(m+1)} = L^{(m)}( y )$.
Using the chain rule, we see that
\begin{equation}\label{eqn:md-recursion}
   y^{(m+1)} = L^{(m)}( y ) 
             = \pd{L^{(m-1)}}{y} \dot{y} 
             = \pd{L^{(m-1)}}{y} L(y), \quad m \in \Z_{\geq1}.
\end{equation}
We note that these functions can be computed analytically for ODEs, especially
given access to symbolic differentiation software.  For PDEs, these higher
derivatives will require the use of the Cauchy Kowalewski procedure
from equation
\eqref{eqn:cauchy-kowalewski}, together with definitions for higher-order
spatial derivatives.  It will be shown in \S\S\ref{subsec:weno-md} and
\ref{subsec:dg-md} that WENO and DG make use of very
different techniques for defining these higher time derivatives.

\begin{defn} \label{def:multideriv-rk}
Given a collection $\{ a^{(m)}_{ij}, b^{(m)}_i \}$ of real scalars,
a multiderivative Runge-Kutta  scheme with $s$-stages and $r$-derivatives
is any update of the form
\begin{equation}
  y_{n+1} = y_n + \sum_{m=1}^r \dt^m \sum_{i=1}^s b^{(m)}_i L^{(m-1)}( y^{(i)} ),
\end{equation}
where intermediate stage values are given by
\begin{equation}
  y^{(i)} = y_n + 
  \sum_{m=1}^r \dt^m \sum_{j=1}^s a^{(m)}_{ij} L^{(m-1)}(y^{(j)}),
\end{equation}
and the total time derivatives of $L$ are given by equation
\eqref{eqn:md-recursion}.
If $a^{(m)}_{ij} = 0$ whenever $i \leq j$, the method is explicit, otherwise it is
implicit.
\end{defn}

We remark that both Taylor and traditional Runge-Kutta
methods are special cases Definition \ref{def:multideriv-rk}:
setting $r=1$ produces traditional Runge-Kutta methods,
and setting $s=1$ produces the Taylor class of methods, with no degrees of
freedom for choosing the $b_i^{(m)}$.
In Table \ref{table:butcher_tableau}, we present the complete Butcher tableau 
for a multiderivative Runge-Kutta method, and in Table
\ref{table:butcher_tableau_taylor},
we present the Butcher tableau for the $r^{th}$-order explicit Taylor method.

\begin{table}
\centering
\normalsize
\caption{Butcher tableau for a multiderivative Runge-Kutta method.
Each $c_i, a_{ij}^{(m)}$ and $b_i^{(m)}$ are real coefficients that define the
method using Definition \ref{def:multideriv-rk}.
For simplicity, we assume time independence, and so the
$c_i$ play no factor in the discretization.  In addition,
we have
$c_i = c_i^{(m)}$, which in general, does not need to be the case.
Classical Runge-Kutta methods are special cases of this form, where
$r=1$, and $a^{(0)}_{ij} = a_{ij}$ where the $a_{ij}$ are the coefficients for
the (single-derivative) Runge-Kutta method.
Explicit Taylor methods have no degrees of freedom, nor stages.
\label{table:butcher_tableau}
}
\begin{tabular}{c|ccc|c|ccc}
$c_1$ & $a_{11}^{(1)}$ & $\cdots$ & $a_{1s}^{(1)}$ 
     & 
     & $a_{11}^{(r)}$ & $\cdots$ & $a_{1s}^{(r)}$ \\
$\vdots$ & $\vdots$ & $\ddots$ & $\vdots$ & $\cdots$ & & $\ddots$ & \\
$c_s$ & $a_{s1}^{(1)}$ & $\cdots$ & $a_{ss}^{(1)}$ 
     & 
     & $a_{s1}^{(r)}$ & $\cdots$ & $a_{ss}^{(r)}$ \\
\hline
& $b_1^{(1)}$ 
& $\cdots$ 
& $b_s^{(1)}$  
& $\cdots$ 
& $b_1^{(r)}$ 
& $\cdots$ 
& $b_s^{(r)}$  
\end{tabular}
\end{table}

\begin{table}
\centering
\normalsize
\caption{Butcher tableau for the explicit Taylor method.
Here, we present the Butcher coefficients for the $r^{th}$-order explicit
Taylor method:
$y_{n+1} = y_n + \sum_{m=1}^r \dt^m L^{(m-1)}(y_n)$, where the $L^{(m-1)}$
describe total time derivatives of equation \eqref{eqn:ode} given by
equation \eqref{eqn:md-recursion}.
Note that there are no degrees of freedom for choosing the $b_i^{(m)}$, 
because they 
are prescribed by $b_1^{(m)} = 1/m!$.
\label{table:butcher_tableau_taylor}
}

\begin{tabular}{c|c|c|c|c|c|c}
   $0$ & & & & & & \\
   \hline
   & $1$ & $1/2!$ & $\cdots$ & $1/m!$ & $\cdots$ & $1/r!$
\end{tabular}
\end{table}

Our definition is an equivalent, yet distinctly different version
of what can be found in other sources (c.f. \cite{HaWa09}).
It is possible to  define
intermediate stages through defining and saving $L^{(m)}( y^{(i)} )$, 
but we prefer Definition \ref{def:multideriv-rk} 
because of the potential for a low storage implementation, at the cost of 
recomputing previously observed values.
For hyperbolic conservation laws, we consider methods that use at most
two-derivatives to be the most portable given that users must have access to
the eigen-decomposition of their problem.

\begin{defn} 
\label{def:two-deriv-rk}
Given a collection $\{ a^{(1)}_{ij}, a^{(2)}_{ij}, b^{(1)}_i  b^{(2)}_i \}$ 
of real scalars,
an $s$-stage, two-derivative Runge-Kutta (TDRK) scheme is any update of the form
\begin{equation}
\label{eqn:explicit-method-two-derivative-update}
y_{n+1} = y_n +  \dt  \sum_{i=1}^s b^{(1)}_i L( y^{(i)} ) 
+ \dt^2 \sum_{i=1}^s b^{(2)}_i \dot{L}( y^{(i)} ),
\end{equation}
where intermediate stage values are given by
\begin{equation}
\label{eqn:explicit-method-two-derivative-stage}
y^{(i)} = y_n + \dt   \sum_{j=1}^s a^{(1)}_{ij} L( y^{(j)} ) 
+ \dt^2 \sum_{j=1}^s a^{(2)}_{ij} \dot{L}( y^{(j) }).
\end{equation}
If $a^{(m)}_{ij} = 0$ for all $i \leq j$, the method is explicit.  
\end{defn}

Before presenting examples of methods from this class, we would like to
draw some comparisons 
between the popular special cases of the multistep multistage 
multiderivative methods.
Our aim is to discuss advantages each method has for being coupled
with numerical PDE solvers, and in particular, we would like to focus on which
methods have promise for working well with new computer architectures.


Traditional Runge-Kutta methods are far and wide the most popular
for solving hyperbolic conservation laws, yet we see room for improvement
given the current direction of computer architecture.
Runge-Kutta methods are easy to implement, and therefore, they are
the most portable of all multistage multiderivative methods.
They are self-starting and can easily change their time step size, which
is an important characteristic to have for solving hyperbolic conservation laws.
In addition, when compared with their natural counterpart, 
the Adams family of methods (e.g. linear multistep methods),
Runge-Kutta methods 
have stability regions that are more favorable
for hyperbolic problems.
For example, on a purely oscillatory problem,
the maximum stable time step for classical fourth-order Runge-Kutta is given by
$| z | \leq \sqrt{8} \approx  2.8$, where $z = \lambda \dt$ is
purely imaginary.
For the same cost and identical storage, one would be able to take four time
steps with fourth-order Adams Bashforth.
Even after rescaling, the maximum stable time step for the equivalent
Adams method would be restricted to $| z | \lessapprox  1.72$.
It would seem that Runge-Kutta methods are ideally suited for solving
hyperbolic conservation laws.  They can be derived to require low-storage
\cite{Williamson80,Ke08,Ke10,NieDieBu12}, can be designed to acquire
strong stability preserving (SSP) properties 
\cite{GoShuTa01,Go05,Ke08}, and are very portable, especially given that they
are self starting.
However, traditional Runge-Kutta methods are not optimal with their 
memory usage,
and to date, even the low-storage Runge-Kutta methods require many stages,
and therefore they may require considerable communication overhead when
compared to pure Taylor schemes.

Taylor methods lie on the other extreme of the multistage multiderivative
methods:
we claim that they can be implemented to have optimally low-storage for
hyperbolic problems,
and can contain minimal communication overhead.
However, pure Taylor methods are the least portable of the time integrators discussed here.
In order to implement a high-order Taylor (e.g. Lax-Wendroff) method for solving
a PDE, one needs to have access to high derivatives of the unknown, which
puts them out of reach from many scientists.
We recognize that this can certainly be done for very complicated 
problems \cite{TaDuBaDiMu07}, but it is difficult to convince users of legacy 
codes to modify them in order to reach high-order accuracy.
On the plus side, high-order Taylor methods 
contain favorable stability regions for hyperbolic
conservation laws, and given that they're single-step methods, they
have nominal communication overhead.  However, the only degrees of freedom
allowed when choosing these methods is the spatial discretization, given that
the time coefficients come directly from the Taylor series.

Given that multiderivative Runge-Kutta methods are a generalization of
traditional Runge-Kutta and pure Taylor methods, it is possible to design
methods from this class that can retain
desirable qualities from each sub-class.
In order to retain portability, we view multiderivative Runge-Kutta methods
that use at most two-derivatives 
as optimal for hyperbolic conservation laws, especially given that most
codes already have access to, or at least users would be willing to implement
the Jacobian of the flux function.
Beyond two-derivatives, we would argue the ``many''-derivative Runge-Kutta 
methods start to lose their portability.
However, given the large size of this class, there is much room for
investigation into what methods work ``best'' with modern architectures.

\subsection{Multistage multiderivative methods: building blocks and examples}
\label{subsec:ode-integrators-example}

Our aim is to describe how to take a multiderivative method from the
ODE literature and formulate a hyperbolic solver using that method.  
In this subsection, we describe a simple building block that can be generalized
to accommodate all explicit multistage multiderivative methods.

The building block we will focus on for the remainder of this paper is given by
the following:
\begin{equation} 
\label{eqn:md-stage}
       y = y_n + 
           \left( \alpha   {\dt} L(y_n) + \beta   \dt^2 \dot{L}(y_n) \right) + 
           \left( \alpha^* {\dt} L(y^*) + \beta^* \dt^2 \dot{L}(y^*) \right).
\end{equation}
In this equation, $y$ could be the full update, as in
$y = y_{n+1}$ from 
equation \eqref{eqn:explicit-method-two-derivative-update}
or a stage value
$y = y^{(i)}$ from equation
\eqref{eqn:explicit-method-two-derivative-stage}.
The key to using this equation to solve PDEs is to provide a definition for
$L$ and $\dot{L}$.

We prefer introducing $\alpha$ and $\beta$ over 
$a_{ij}^{(m)}$ and $b_{i}^{(m)}$ from Definitions 
\ref{def:multideriv-rk} and
\ref{def:two-deriv-rk} because
these letters delete unnecessary indices and the upcoming
descriptions for the PDE methods will introduce further indices that would
become cumbersome.

\begin{rmk}Extensions to multistage, `many'-derivative methods follow by adding 
extra terms to equation \eqref{eqn:md-stage}.  
\end{rmk}

More stages require more terms to be added to \eqref{eqn:md-stage}.  For
example, a three stage, two-derivative method is entirely defined after defining
updates of the form:
\begin{equation}\label{eqn:md-stage-td-3stage}
\begin{aligned}
   y = y_n &+ \left( \alpha \dt L(y_n) + \beta \dt^2 \dot{L}(y_n) \right)
            + \left( \alpha^* \dt L(y^*) + \beta^* \dt^2 \dot{L}(y^*) \right) \\
           &+ \left( \alpha^{**} \dt L(y^{**}) + \beta^{**} \dt^2 \dot{L}(y^{**}) \right)
\end{aligned}
\end{equation}
for arbitrary values of $\alpha$ and $\beta$.  Again, the $y$ in this equation
can be a single stage value 
$y = y^{(i)}$ as in equation
\eqref{eqn:explicit-method-two-derivative-stage}, or a full update, as in 
equation \eqref{eqn:explicit-method-two-derivative-update}.

We point out that three-derivative,
two-stage methods can be formulated with
\begin{equation}\label{eqn:md-stage-taylor3}
\begin{aligned}
   y = y_n &+ \left( \alpha \dt L(y_n) + \beta \dt^2 \dot{L}(y_n) + \gamma \dt^3
                     \ddot{L}( y_n )\right) \\
           &+ \left( \alpha^* \dt L(y^*) + \beta^* \dt^2 \dot{L}(y^*) +
                     \gamma^* \dt^3 \ddot{L}( y^* )\right).
\end{aligned}
\end{equation}
Note that setting $\alpha = 1$, $\beta = 1/2$ and $\alpha^* = \beta^* = 0$ in equation \eqref{eqn:md-stage} 
produces the second-order Taylor method, and setting 
$\alpha = 1$, $\beta = 1/2$, $\gamma = 1/6$ and
$\alpha^* = \beta^* = \gamma^* = 0$ in equation
\eqref{eqn:md-stage-taylor3} produces the third order Taylor method.

\subsubsection{Multistage multiderivative methods: some examples}

We now describe how equation \eqref{eqn:md-stage} can be used to construct
multiderivative methods.  We assert that these methods have not necessarily been
optimized for hyperbolic problems; our chief objective is to demonstrate how to
implement these methods.  An investigation into optimized schemes will be
pursued in the future.

A third order, two-stage, two-derivative method (TDRK3) \cite{ChTs10} is given by:
\begin{equation}\label{eqn:tdrk3}
\begin{aligned}
   y^{*}   &= y_n + \dt L(y_n) + \frac{ (\dt)^2 }{ 2 } \dot{L}( y_n ) , \\
   y_{n+1} &= y_n + \dt \left( \frac2 3 L(y_n) + \frac1 3 L(y^*) \right) 
                  + \frac{\dt^2}{6} \dot{L}( y_n ).
\end{aligned}
\end{equation}
This method can be constructed by first inserting
\[\alpha = 1, \quad \beta = 1/2, \quad \alpha^* = \beta^* = 0,\] 
into equation \eqref{eqn:md-stage} to construct 
the intermediate stage, and the final update is given by selecting
\[\alpha = 1, \quad \beta = 1/6, \quad \alpha^* = 0 \quad \text{and} \quad \beta^*
= 1/3.\]
The Butcher tableau for this method is provided in Table \ref{table:tdrk3}, and
the region of absolute stability is plotted in Figure \ref{fig:two-derivative-stability},
which is identical to any three stage classical Runge-Kutta method.

\begin{table}
\normalsize
\centering
\caption{Butcher tableau for a third-order two-derivative method.
Presented here are the coefficients as in Table 
\ref{table:butcher_tableau} for an explicit, third-order method \cite{ChTs10}.
Note that all diagonal and upper-triangular
entries are zero, meaning that the scheme is explicit.
}
\begin{tabular}{c|cc|cc}
$0$   & 0 & 0 & 0 & 0 \\
$1$ & $1/2$ & 0 &
$1/8$ & $0$ \\
\hline
& $2/3$ & $1/3$ & $1/6$ & $0$
\end{tabular}
\label{table:tdrk3}
\end{table}

\begin{figure}
\normalsize
\centering
\includegraphics[width=1.0\textwidth]{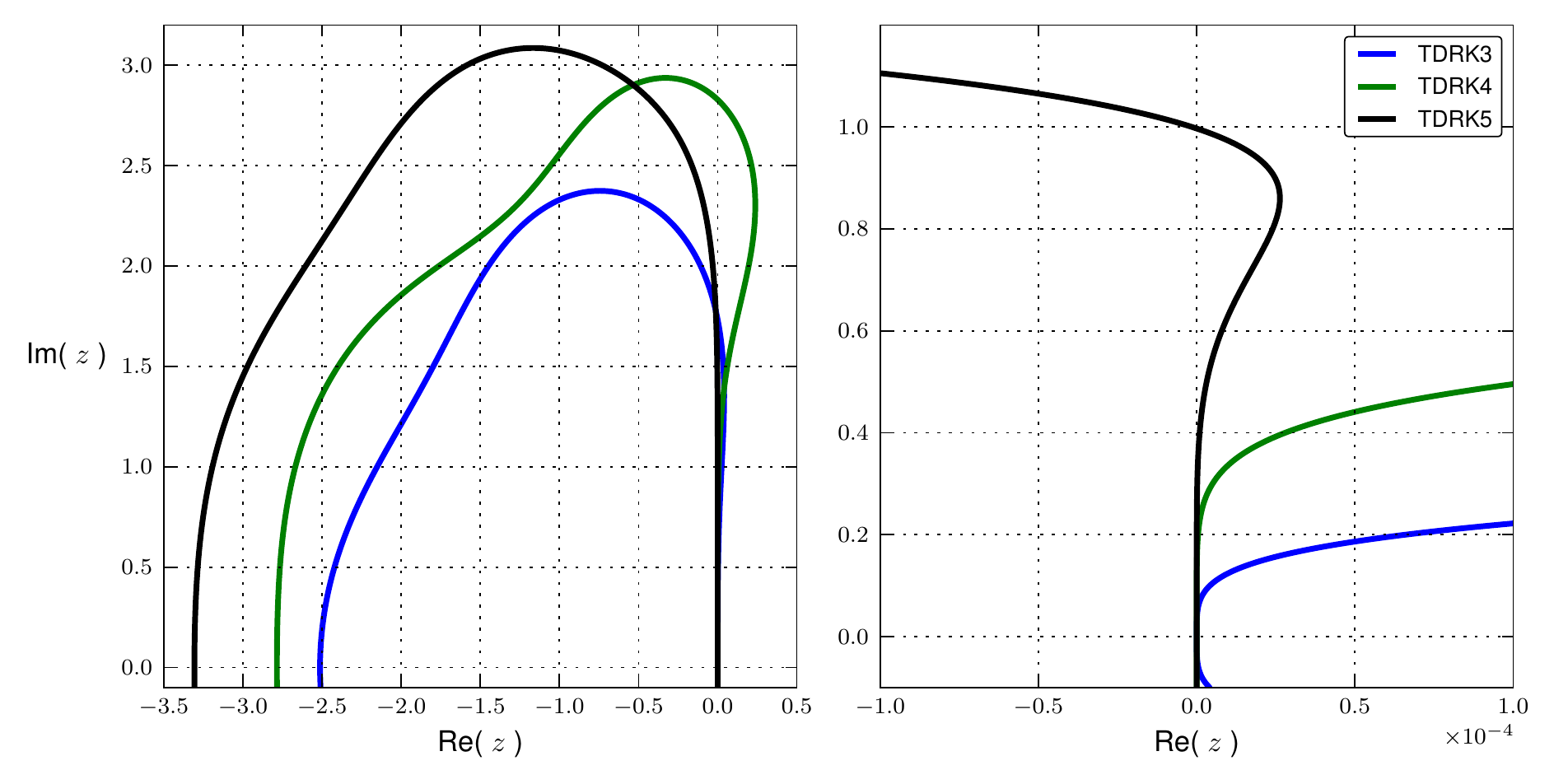}
\caption{Regions of absolute stability.
Here, we plot the regions of absolute stability for three different
two-derivative methods that are derived in Chan and Tsai \cite{ChTs10}: 
TDRK3 \eqref{eqn:tdrk3}, TDRK4 \eqref{eqn:exp4} and TDRK5 \eqref{eqn:tdrk5}
which are in order of smallest to largest.
The picture on the right is a zoomed in picture of the imaginary axis.
Note that the third and fourth-order methods have regions of absolute
stability identical to classical three and four stage, respectively, RK
methods.  Of particular importance for hyperbolic problems
is the fact that each of these integrators
contain part of the imaginary axis \cite{bLe02}.
}
\label{fig:two-derivative-stability}
\end{figure}

A fifth-order, three-stage, two-derivative method \cite{ChTs10} is given by:
\begin{equation}\label{eqn:tdrk5}
\begin{aligned}
   y^{*}   &= 
       y_n + \frac{2}{5} \dt L(y_n) + \frac{2}{ 25 } \dt^2 \dot{L}( y_n ) , \\
   y^{**}  &= 
       y_n + \dt L(y_n) + \dt^2 
       \left( -\frac{1}{4} \dot{L}( y_n ) + \frac{3}{4} \dot{L}( y^* ) \right) , \\
   y_{n+1} &= 
       y_n + \dt L(y_n) + \dt^2 \left( \frac{1}{8} \dot{L}( y_n ) +
       \frac{25}{72} \dot{L}(y^*) + \frac{1}{36} \dot{L}(y^{**}) \right).
\end{aligned}
\end{equation}
This method can be constructed by first inserting
\[  \alpha = 2/5, \quad \beta = 2/25, \quad 
    \alpha^* = \beta^* = 0, \quad
    \alpha^{**} = \beta^{**} = 0\] 
into equation 
\eqref{eqn:md-stage-td-3stage}
to produce $y^*$,  followed by inserting
\[  \alpha = 1, \quad \beta = -1/4, \quad 
    \alpha^* = 0, \quad \beta^* = 3/4, \quad
    \alpha^{**} = \beta^{**} = 0,\] 
into \eqref{eqn:md-stage-td-3stage} to produces a third stage, $y^{**}$.
The final update is then given by selecting
\[\alpha = 1, \quad \beta = 1/8, \quad \alpha^* = \alpha^{**} = 0, \quad
\beta^* = \frac{25}{72}
    \quad \text{and} \quad \beta^{**} = \frac{1}{36}. \]
The Butcher tableau for this method is provided in Table \ref{table:tdrk5},
and the region of absolute stability is plotted in Figure \ref{fig:two-derivative-stability}.

\begin{table}
\normalsize
\centering
\caption{Butcher tableau for a fifth-order, three-stage, two-derivative method.
Presented here are the coefficients as in Table 
\ref{table:butcher_tableau} for an explicit, fifth-order method \cite{ChTs10}.
Note that all diagonal and upper-triangular
entries are zero, meaning that the scheme is explicit.
We remark that this scheme has not necessarily been optimized
given that some entries were zeroed out by choice in order to reduce the
complexity of the order conditions.  However, 
we present this example given that it contains a favorable stability region
because it contains part of the imaginary axis (c.f. Figure
\ref{fig:two-derivative-stability}).
\label{table:tdrk5}
}
\begin{tabular}{c|ccc|ccc}
$0$   & 0 & 0 & 0 & 0 & 0 & 0\\
$2/5$ & $2/5$ & 0 & 0 & $2/25$ & 0 & 0\\
$1$   & $1$ & 0 & 0 & $-1/4$ & $3/4$ & 0\\
\hline
& $1$ & 0 & 0 & $1/8$ & $25/72$ & $1/36$
\end{tabular}
\end{table}

\subsubsection{Multistage multiderivative methods: the canonical example}

The example used for the remainder of this work will now be presented.
There is a unique combination for an $s=2$-stage method that
produces fourth-order accuracy \cite{KaWa72-RK,OnYo04,ChTs10}.  
This method, which we refer to as TDRK4, can be
written as
\begin{equation}\label{eqn:exp4}
\begin{aligned}
   y^{*}   &= y_n + \frac{\dt}{2} L_n 
                  + \frac{ (\dt/2)^2 }{ 2 } \dot{L}( y_n ) , \\
   y_{n+1} &= y_n + \dt L_n + \frac{\dt^2}{2} 
       \left[ \frac1 3\left( \dot{L}( y_n )  + 2 \dot{L}( y^* ) \right) \right],
\end{aligned}
\end{equation}
The Butcher tableau for this method is presented in
Table \ref{table:tdrk4}, and the region of absolute stability is plotted in
Figure \ref{fig:two-derivative-stability}.    Note that the region of absolute stability
is identical to any four-stage, fourth-order Runge-Kutta method.

We choose to use this method as our canonical example for three reasons.  First,
this example is the simplest scheme that does not fall under the Taylor or
Runge-Kutta umbrella, and therefore serves as a demonstrative example of new
methods that can be found from this class.  Second, 
this method has been optimized for low storage and high-order accuracy given 
two stages, and two-derivatives, and thirdly, this method works well for
hyperbolic problems.

\begin{table}
\normalsize
\centering
\caption{Butcher tableau for the multiderivative method investigated in this
work.  Presented here are the coefficients as in Table 
\ref{table:butcher_tableau} for the explicit, fourth-order method, TDRK4 presented
in equation \eqref{eqn:exp4}.  Note that all diagonal and upper-triangular
entries are zero, meaning that the scheme is explicit.
\label{table:tdrk4}
}
\begin{tabular}{c|cc|cc}
$0$   & 0 & 0 & 0 & 0 \\
$1/2$ & $1/2$ & 0 &
$1/8$ & $0$ \\
\hline
& $1$ 
& $0$
& $1/6$
& $1/3$
\end{tabular}
\end{table}

Given that our goal is to describe how to implement the large class of
explicit multistage multiderivative methods for solving PDEs,
for simplicity of exposition, we use this method as 
our canonical example of a method from this class.
In addition, we claim that a complete description for this scheme will provide 
the necessary
mechanisms for extension and investigation into integrators containing extra
stages.  These integrators can be constructed to be even higher 
order accurate and contain
favorable stability regions.  For example, all of the methods presented in Chan
and Tsai \cite{ChTs10} can be implemented using our description with a 
straight-forward extension of what follows.

Observe that the TDRK4 method in equation \eqref{eqn:exp4}
can be constructed from of equation \eqref{eqn:md-stage}, 
with the first stage
given by 
\[\alpha = 1/2, \quad \beta = 1/8, \quad \alpha^* = \beta^* = 0,\] 
and the update given by 
\[\alpha = 1, \quad \beta = 1/6, \quad \alpha^* = 0 \quad \text{and} \quad \beta^*
= 1/3.\]

This completes our description of the multiderivative Runge-Kutta scheme, 
and it bears repeating that without loss of generality, we
will use method \eqref{eqn:exp4} as our canonical example of a method from this
class.
Given that the focus of this work is how to implement this scheme for solving
hyperbolic PDEs, we still need to describe how to discretize in space.

\section{The finite difference WENO method}
\label{sec:weno}

The finite difference weighted essentially non-oscillatory (WENO) method 
has many variations and a long history of development.
The original method was developed by Shu and Osher \cite{ShuOsher87,ShuOsher89},
and later analyzed and further developed by Shu and his collaborators 
\cite{JiangShu96,Shu97,Shu01}.  
For a recent comprehensive review of the many variations of WENO schemes, 
see Chi-Wang Shu's extensive review paper \cite{Shu09}.
In this work, we consider the fifth-order WENO-Z scheme 
\cite{HeAsPo05,BoCaCoWai08,CaCoDo11}, that is an improvement of the classical
fifth-order Jiang and Shu (WENO-JS) scheme \cite{JiangShu96}.
The underlying choice of the reconstruction procedure is not central to this
work and our description will be generic enough to accommodate most 
variations of the classical WENO method.
However, given that there are many options, we explain the minimal details 
necessary to reproduce the present work.

In \S\ref{subsec:weno-fd} we present the classical MOL formulation for the
WENO method.  In \S\ref{subsec:weno-md} we describe the extension of
multiderivative ODE integrators presented in
\S\S\ref{subsec:ode-integrators-definitions}--\ref{subsec:ode-integrators-example},
to formulate the multiderivative
WENO method that is the subject of this section.
In \S\ref{subsec:dg-md} we
will see that the extension of multiderivative ODE methods from 
\S\ref{sec:methods}
to the DG method requires a very different approach.

\subsection{The finite difference WENO method: MOL formulation}
\label{subsec:weno-fd}
%
%
We begin our description with a reduction of 
\eqref{eqn:hyperbolic-law} to a 
1D conservation law:
\begin{equation}
\label{eqn:1dhyper}
   q_{,t} + f(q)_x =0, \quad q(0,x) = q_0(x), \quad x \in \Omega = [a,b].
\end{equation}
Here, we follow the 
standard convention of naming our flux function $f$ in place of $R$.
Moreover, \eqref{eqn:1dhyper} is \emph{hyperbolic} if the Jacobian
$f'(q) = R \Lambda R^{-1}$ is diagonalizable with real 
eigenvalues for all $q$ in the domain of interest.  

The spatial discretization for a finite difference method seeks a point-wise
approximation to the exact solution of \eqref{eqn:1dhyper}
at a finite collection of points.
We start with a uniform discretization of $[a,b]$ using $m_x$ points:
\begin{equation}
   \dx = (b-a) / m_x, \quad
   \quad x_i = a + (i-1/2) \dx, \quad
   \quad i\in \left\{ 1, 2, \ldots, m_x \right\},
\end{equation}
and we seek values $q_i$ that approximate the exact solution at each grid point,
$q_i(t) \approx q(t, x_i)$.

In order to write \eqref{eqn:1dhyper} in discrete flux-difference form so we
can have a conservative method\footnote{A finite difference method is
\emph{conservative} if the method satisfies $\frac{d}{dt}\left( \sum_i q_i(t)
\right) = 0$ on a periodic (or infinite) domain.},
we begin by defining an implicit \emph{sliding function} $h$ through
\begin{equation}
   f\left( q(t, x) \right) 
   = \frac{1}{\dx} \int_{x - \dx/2}^{x + \dx/2 } h(t, x)\, dx.
\end{equation}
With this definition in place, we have the nice result that 
\begin{equation}
   \frac{\partial f}{ \partial x}\left( q(t,x_i) \right) = \frac{1}{\dx}
   \left[ h\left(t, x_{i+1/2}\right) - h\left(t, x_{i-1/2} \right) \right],
\end{equation}
which is precisely what is needed to define a discrete flux difference
formulation of $q_t = -f_x$.
The MOL formulation for the finite difference scheme defines
interpolated values $h_{i+1/2}$ that approximate $h\left(t, x_{i+1/2}\right)$
to high-order accuracy,
\begin{equation}
   \frac{1}{\dx} \left( h_{i+1/2} - h_{i-1/2} \right) = 
   \frac{\partial f}{ \partial x}\left( q(t,x_i) \right) 
   + \BigO\left( \dx^M \right).
\end{equation}
Moreover, an astute observation means that $h$ need never be computed
\cite{Shu97,Shu09}:
define cell averages of $h(t,x)$ as
\begin{equation}
   \h_i := \frac 1 \dx \int_{x_{i-1/2}}^{ x_{i+1/2} } h(t, x)\, dx,
\end{equation}
then observe that $f( q_i(t) ) = \h_i$, and therefore point
values of $f$ can be interpreted as cell averages of $h$.
After defining an appropriate interpolating algorithm for producing
high-order interface values $h_{i+1/2}$ from cell averages $\h_i$,
the full MOL formulation is given by,
\begin{equation}\label{eqn:fd-L}
   \frac{d}{dt} q_i( t ) = -\frac{1}{\dx} \left( h_{i+1/2} - h_{i-1/2} \right).
\end{equation}
This scheme is automatically 
conservative as it is written in flux difference form.
Usually one applies a high-order explicit Runge-Kutta integrator to
\eqref{eqn:fd-L}, which results in what
is normally called the Runge-Kutta WENO (RK-WENO) 
method.

\subsubsection{The finite difference WENO method: the reconstruction procedure}
\label{subsubsec:weno-weights}

A conservative reconstruction procedure requires a single value $h_{i+1/2}$
for equation
\eqref{eqn:fd-L}
to be defined at each grid interface.
In the ensuing discussion, we suppress the time dependence of $h$, and assume
that we have known cell averages
$\h_i$ for a function $h = h(x)$.
The fifth-order WENO method uses
a five point stencil shifted to the left or right of
the interface:
\begin{align*}
   h_{i+1/2}^+ &:= WENO5^+[ \h_{i-2}, \h_{i-2}, \h_{i-1}, \h_i, \h_{i+2} ], \\
   h_{i+1/2}^- &:= WENO5^-[ \h_{i-1}, \h_{i},   \h_{i+1}, \h_{i+2}, \h_{i+3} ].
\end{align*}
Here, we define coefficients for the function $WENO5^+$, and note that 
by symmetry, the 
reconstruction procedure weighted in the other direction can be 
observed by flipping the stencil:
\begin{align}
   WENO5^-[ \h_{i-1}, \h_{i},   \ldots \h_{i+3} ] :=
   WENO5^+[ \h_{i+3}, \h_{i+2}, \ldots \h_{i-1} ].
\end{align}

Three sub-stencils 
$S_0 = \{ \h_{i-2}, \h_{i-1}, \h_{i  } \}$,
$S_1 = \{ \h_{i-1}, \h_{i  }, \h_{i+1} \}$, and
$S_2 = \{ \h_{i  }, \h_{i+1}, \h_{i+2} \}$
uniquely 
define three quadratic polynomials $p_j(x)$
that have the same cell
averages for each element in their stencil.  
Each polynomial defines a competing
value for $h(x_{i+1/2})$ with $h^{(j)}_{i+1/2} = p_j( x_{i+1/2} )$:
\begin{subequations}
\begin{align}
\label{eqn:weno-interp0}
   h^{(0)}_{i+1/2} &= \phantom{-} 
   \frac{1}{3} \h_{i-2} -  \frac{7}{6} \h_{i-1} + \frac{11}{6} \h_i, \\
\label{eqn:weno-interp1}
   h^{(1)}_{i+1/2} &= - 
   \frac{1}{6} \h_{i-1} +  \frac{5}{6} \h_{i\phantom{-1}} + \frac1 3 \h_{i+1}, \\
\label{eqn:weno-interp2}
   h^{(2)}_{i+1/2} &= \phantom{-} 
   \frac{1}{3} \h_{i\phantom{-1}} +  \frac{5}{6} \h_{i+1} - \frac{1}{6} \h_{i+2}.
\end{align}
\end{subequations}

%
The \emph{linear weights} 
$\gamma_j = \left\{1/10,\, 3/5,\, 3/10\right\}$ 
are defined as the unique linear combination 
of equations \eqref{eqn:weno-interp0}-\eqref{eqn:weno-interp2} 
that yields a fifth-order accurate point value for $h(x_{i+1/2})$:
\begin{align}
   h_{i+1/2} = \gamma_0 h^{(0)}_{i+1/2} +
               \gamma_1 h^{(1)}_{i+1/2} +
               \gamma_2 h^{(2)}_{i+1/2}.
\end{align}
%
The WENO procedure replaces the linear weights $\gamma_j$ with nonlinear 
weights $\omega_j$ that are necessary in regions with strong shocks.
The Jiang and Shu \emph{smoothness indicators} $\beta_j$ place a 
quantitative measure
on the smoothness of each stencil based on a Sobolev norm:
\begin{align}
   \beta_j := \sum_{l=1}^k \dx^{2l-1} \int_{x_{i-1/2}}^{x_{i+1/2}} 
   \left( \frac{d^l}{dx^l} p_j(x) \right)^2\, dx,
\end{align}
which for our fifth order method, are given by
\begin{align}
\begin{aligned}
   \beta_0 &=  \frac{13}{12} \left(  \h_{i-2}- 2\h_{i-1} +  \h_i     \right)^2
             + \frac{1}{4}   \left(  \h_{i-2}- 4\h_{i-1} + 3\h_i     \right)^2, \\
   \beta_1 &=  \frac{13}{12} \left(  \h_{i-1}- 2\h_{i}   +  \h_{i+1} \right)^2
             + \frac{1}{4}   \left(  \h_{i-1}-  \h_{i+1}             \right)^2, \\
   \beta_2 &=  \frac{13}{12} \left(  \h_{i}  - 2\h_{i+1} +  \h_{i+2} \right)^2
             + \frac{1}{4}   \left( 3\h_{i}  - 4\h_{i+1} +  \h_{i+2} \right)^2.
\end{aligned}
\end{align}
The \emph{non-linear} WENO-Z weights $\omega^z_j$ are a slight modification of the
classical weights $\omega_j$.  The new weights require the computation of
a single additional parameter
$\tau^5 = \left|\beta_2-\beta_0\right|$:
\begin{align}
   \omega^z_k = \frac{\tilde{\omega}^z_k}{ \sum_{l=0}^2 \tilde{\omega}^z_l }, 
       \quad
   \tilde{\omega}^z_k = \frac{\gamma_k}{\beta^z_k}, 
       \quad 
   \beta^z_k =  1 + \left( \frac{\tau^5}{\beta_k+\eps} \right)^p.
\end{align}
We use the power parameter $p=2$ and regularization parameter
$\eps = 10^{-12}$ for all of our simulations.
With these definitions in place, the final interpolated value is defined as
\begin{align}
 WENO5^+[ \h_{i-2}, \ldots \h_{i+2} ] := 
 \omega^z_0 h^{(0)}_{i+1/2} +
 \omega^z_1 h^{(1)}_{i+1/2} +
 \omega^z_2 h^{(2)}_{i+1/2}.
\end{align}

\subsubsection{The finite difference WENO method: MOL formulation for systems}

For a system of conservation laws, the reconstruction procedure
described in \S\ref{subsubsec:weno-weights} is carried out 
locally on each of the scalar characteristic variables:

\subsection*{\bf Finite difference WENO procedure for 1D systems}

\smallskip

\begin{enumerate}

%
%
\item
For each $i$, evaluate $f_i = f( q_i )$ at each mesh point
and compute average values of $q$ at the half grid points
\begin{align} \label{eqn:roe-avg}
   q^*_{i-1/2} = \frac{1}{2}\left( q_i + q_{i-1} \right).
\end{align}
Roe averages \cite{Roe81} may be used in place of \eqref{eqn:roe-avg},
but all the results presented in this work use this arithmetic average.

\item
Compute the left and right eigenvalue decomposition
of $f'(q) = R \Lambda R^{-1}$ at the half-grid points:
\begin{align}
   R_{i-1/2} = R( q^*_{i-1/2} ), \quad
   R^{-1}_{i-1/2} = R^{-1}( q^*_{i-1/2} ).
\end{align}
Compute $\alpha := \max_i | f'(q^*_{i-1/2}) |$ as the fastest wave speed in
the entire system.  For stability, we follow the common practice of
increasing $\alpha$ by exactly 10\% in
order to completely contain the fastest wave speed,
that is, we set $\alpha = 1.1 \cdot \max_i | f'(q^*_{i-1/2}) |$.
The value $| f'(q) |$ is the maximum absolute value
of all eigenvalues of the Jacobian, $f'(q)$.
\smallskip

\item
For each $i$, determine the weighted ENO stencil  $\{i + r\}$ surrounding $i$.
In fifth order WENO, the full stencil is 
given by $r \in \{-3, -2, -1, 0,1,2 \}$.  Project each $q_{i+r}$ and flux 
values $f_{i+r}$ onto the characteristic variables using $R^{-1}_{i-1/2}$:
\begin{subequations}
\begin{align}
   w_{ i + r } = R^{-1}_{i-1/2} \cdot q_{i + r}, \\
   g_{i + r }   = R^{-1}_{i-1/2} \cdot f_{i + r}.
\end{align}
\end{subequations}
Apply Lax-Friedrichs flux splitting on $g_{i+r}$:
\begin{align}
   g^{\pm}_{i+r} = \frac{1}{2}\left( g_{i+r} \pm \alpha \,  w_{i+r} \right).
\end{align}
As an alternative, one could use a local wave speed.
Note that for each component, these definitions automatically satisfy 
$\frac{d g^+}{dw}  \geq 0$ and
$\frac{d g^-}{dw}  \leq 0$.

\smallskip

\item
Perform a WENO reconstruction on the characteristic variables.  
Use the stencil which uses an extra point on the upwind direction for 
defining $g^\pm$:
\begin{align*}
 \hat{g}^+_{i-1/2} &= 
 WENO5^+ \left[ g^+_{i-3}, g^+_{i-2}, g^+_{i-1}, g^+_{i}, g^+_{i+1} \right], \\
 \hat{g}^-_{i-1/2} &= 
 WENO5^- \left[ g^-_{i-2}, g^-_{i-1}, g^-_{i}, g^-_{i+1}, g^-_{i+2} \right].
\end{align*}
Define $\hat{g}_{i-1/2} := \hat{g}^+_{i-1/2}  + \hat{g}^-_{i-1/2}$.

\item
Using same projection matrix, $R_{i-1/2}$, project characteristic variables 
back onto the conserved variables:
\begin{align}
   \hat{f}_{i-1/2} := R_{i-1/2} \cdot \hat{g}_{i-1/2}.
\end{align}

\end{enumerate}

\subsection{The finite difference WENO method: a multiderivative formulation}
\label{subsec:weno-md}

In order to implement multiderivative methods into the WENO framework,
we closely follow previous work on Lax-Wendroff WENO methods
\cite{QiuShu03}, but in place of relying on a single-step Taylor
series, we build intermediate stages of the form given by \eqref{eqn:md-stage}.
Consider two `stages', $q^n$ and $q^*$ that provide a pointwise approximation
to the exact solution.
If we apply \eqref{eqn:cauchy-kowalewski} to each $q^n$ and $q^*$
and insert the result into \eqref{eqn:md-stage}, 
we see that the starting point for putting together a multiderivative WENO
integrator is to define the following for arbitrary values of 
$\alpha$, $\alpha^*$, $\beta$ and $\beta^*$:
\begin{equation}\label{eqn:md-stage-weno}
\begin{split}
   q = q^n &- \alpha   \dt f\left( q^n \right)_{\!,\,x} 
            - \alpha^* \dt f\left( q^* \right)_{\!,\,x}  \\
           &- \beta   \dt^2 \left( f'(q^n) q^n_{,t} \right)_{\!,\,x}
            - \beta^* \dt^2 \left( f'(q^*) q^*_{,t} \right)_{\!,\,x}.
\end{split}
\end{equation}
Note that we have retained the value $q_{,t}$ in
place of substituting $q_{,t} = -f_{,x}$ into the last two terms.
This is different than what will be done in \S\ref{subsec:dg-md} for the
multiderivative DG method.
The complete multiderivative WENO method is given by the following:

\subsection*{\bf Multistage multiderivative WENO procedure}

\smallskip

\begin{enumerate}

\item
Given two pointwise approximations $q^n$ and $q^*$,
perform a single WENO reconstruction on each piece to construct the following 
values:
\begin{align}
   q^n_{i,t} := 
       -\frac{1}{\dx}\left( \hat{f}^n_{i+1/2} - \hat{f}^n_{i-1/2} \right),
   \quad
   q^*_{i,t} := 
       -\frac{1}{\dx}\left( \hat{f}^*_{i+1/2} - \hat{f}^*_{i-1/2} \right).
\end{align}

\smallskip

\item
Define the pointwise values 
${G}^n_{i} := f'\left( q^n_i \right) q^n_{i,t}$ and
${G}^*_{i} := f'\left( q^*_i \right) q^*_{i,t}$.
Note that we do not need to decompose $q$ onto the 
characteristic variables.

\smallskip

\item
For each $G^n$ and $G^*$ in Step 2, 
compute a finite difference approximation to $G^n_{,x}$ and $G^*_{,x}$.
Given that this term inherits an extra factor of $\dt$, we can use
the fourth-order centered finite difference:
\begin{align}
   D_x G_{i} := \frac{1}{12 \dx}
   \left ( G_{i-2}-8G_{i-1} + 8 G_{i+1} - G_{i+2} \right).
\end{align}
Using a centered stencil means that the method will automatically be 
conservative.  
The complete WENO discretization of \eqref{eqn:md-stage-weno} is now,
\begin{equation}
\label{eqn:md-stage-weno2}
   q_{i} = q_{i}^n 
       + \alpha   \dt q^n_{i,t}
       + \alpha^* \dt q^*_{i,t}
        - \beta   \dt^2 D_{x} G^n_{i}
        - \beta^* \dt^2 D_{x} G^*_{i}.
\end{equation}
This equation defines the building block for creating any two-derivative
Runge-Kutta method, 
because the definitions for $L$ and $\dot{L}$ from Definition
\ref{def:two-deriv-rk} are now in place.
\end{enumerate}

We now summarize the entire multiderivative WENO procedure.  Given a 
multiderivative Runge-Kutta scheme, we use a WENO reconstruction procedure 
to define $q_{,t}$.
Higher derivatives $q_{,tt}$, etc. are defined by 
first using the Cauchy-Kowalewski procedure in \eqref{eqn:cauchy-kowalewski}
to define exact formulas for these terms.  The spatial derivatives that show up
in these terms are then discretized by using finite differences.
Once $q_{,t}, q_{,tt}$, etc., have been defined, one uses the coefficients from
the multiderivative scheme to construct stages as well as a final update.  
This entire
procedure generalizes both the Lax-Wendroff and RK-WENO
methods.  For completeness, we prove that the proposed scheme is conservative
by showing that equation \eqref{eqn:md-stage-weno2} is conservative
for any choice of $\alpha$, $\alpha^*$, $\beta$ and $\beta^*$ that are defined
by the selected multiderivative scheme.

\begin{thm} \label{thm:weno-mass-conservation}
The proposed multistage multiderivative WENO scheme is mass conservative.
\end{thm}
\begin{proof}
The final update, $q^{n+1}_i$ is given by equation
\eqref{eqn:md-stage-weno2} 
for a collection of coefficients $\alpha$, $\alpha^*$, $\beta$ and $\beta^*$
that depend on the selected scheme.  Summing over $i$ produces
\begin{subequations}
\begin{align}
\sum_{i} q^{n+1}_i &= \sum_i \left( q^n_i + \alpha   \dt q^n_{i,t}
       + \alpha^* \dt q^*_{i,t}
        - \beta   \dt^2 D_{x} G^n_{i}
        - \beta^* \dt^2 D_{x} G^*_{i}\right) \\ \label{eqn:weno-conservationb}
        &= \sum_i q^n_i + \alpha \dt \sum_i q^n_{i,t}
        + \alpha^* \dt \sum_i q^*_{i,t}
        + \beta   \dt^2 \sum_i D_x G_{i}
        + \beta^* \dt^2 \sum_i D_x G^*_{i}.
\end{align}
\end{subequations}
Conservation will follow after showing that the last four terms in equation
\eqref{eqn:weno-conservationb} are zero.  First, observe that
\begin{subequations}
\begin{align}
\sum_i q^n_{i,t} &= -\frac{1}{\dx}\sum_i \left( \hat{f}^n_{i+1/2} - \hat{f}^n_{i-1/2} \right)
= -\frac{1}{\dx} \left( \sum_i \hat{f}^n_{i+1/2} - \sum_i \hat{f}^n_{i-1/2} \right) \\
&= -\frac{1}{\dx} \left( \sum_i \hat{f}^n_{i+1/2} - \sum_i \hat{f}^n_{i+1/2} \right)
= 0.
\end{align}
\end{subequations}
Similarly, $\sum_i q^*_{i,t} = 0$.  The last two terms sum to zero because
a central difference stencil is used:
\begin{subequations}
\begin{align}
\sum_i D_x G^n_{i} &= 
\frac{1}{12 \dx} \sum_i 
   \left ( G^n_{i-2}-8G^n_{i-1} + 8 G^n_{i+1} - G^n_{i+2} \right) \\
   &= \frac{1}{12\dx} \left( 
	\sum_i G^n_{i-2}-8 \sum_i G^n_{i-1} + 8 \sum_i G^n_{i+1} - \sum_i G^n_{i+2} \right) \\
   &= \frac{1}{12\dx} \left( 
	\sum_i G^n_{i}-8 \sum_i G^n_{i} + 8 \sum_i G^n_{i} - \sum_i G^n_{i} \right) = 0.
\end{align}
\end{subequations}
Similarly, $\sum_i D_x G^*_i = 0$, and therefore 
$\sum_i q^{n+1}_i = \sum_i q^{n}_i$ for all $n$. \qed
\end{proof}

\begin{rmk}{Various multistage multiderivative WENO methods can be built by 
repeated application of \eqref{eqn:md-stage-weno2} with different values of
$\alpha, \alpha^*, \beta$ and $\beta^*$.}
\end{rmk}

The procedure for constructing multiderivative WENO
methods using \eqref{eqn:md-stage-weno2}
is identical to that already presented in 
\S\ref{subsec:ode-integrators-example} for ODEs.
For example, setting $\beta = \beta^* = 0$ reproduces Runge-Kutta methods, and
setting $\alpha^* = \beta^* = 0$ reproduces the second-order Lax-Wendroff WENO
method, provided we define $\alpha = 1$, and $\beta = 1/2$.
The $s=3$-stage analogue of equation \eqref{eqn:md-stage-taylor3} for PDEs
is given by
\begin{align}
\begin{split}
   q_i = q_i^n & + \alpha   \dt q^n_{i,t}
                 + \alpha^* \dt q^*_{i,t}
                 + \alpha^{**} \dt q^{**}_{i,t} \\
               & - \beta     \dt^2 D_{x} G^n_{i}
                 - \beta^*    \dt^2 D_{x} G^*_{i}
                 - \beta^{**} \dt^2 D_{x} G^{**}_{i}.
\end{split}
\end{align}
In addition, further derivatives can be included, but one would need
to revisit \eqref{eqn:cauchy-kowalewski} before inserting these terms.
We would expect that higher derivatives
can be approximated using smaller centered finite difference stencils
because they get multiplied by increasing powers of $\dt$
(c.f. \cite{QiuShu03} for further details).

We repeat that equation \eqref{eqn:md-stage-weno2} finishes 
the spatial discretization
of \eqref{eqn:md-stage}.  If we appeal to this equation twice we can 
construct the
unique two-stage, fourth-order method, TDRK4 that is the PDE analogue of
\eqref{eqn:exp4}:
\begin{subequations}
\begin{align}
q^*_i &= q^n_i -\frac{\dt}{2\dx}\left( \hat{f}^n_{i+1/2} - \hat{f}^n_{i-1/2} \right)
- \frac{(\dt/2)^2}{2} \left( D_x G_i \right); \\
q^{n+1}_i &= q^n_i - \frac{\dt}{\dx}\left( \hat{f}^n_{i+1/2} - \hat{f}^n_{i-1/2} \right)
- \frac{\dt^2}{6}\left( D_x G_i + 2 D_x G^*_i \right).
\end{align}
\end{subequations}

Numerous examples of this WENO scheme
are provided in \S\ref{sec:numerical_examples}, where 
we compare this method with classical fourth-order 
Runge-Kutta (RK4), as well as the third-order 
strong stability preserving (SSP) method
of Shu and Osher \cite{ShuOsher87}.
Before presenting results, 
we first describe an implementation of
a multistage multiderivative discontinuous Galerkin method that will
require a different approach to define the higher temporal derivatives.

\section{The discontinuous Galerkin (DG) method}
\label{sec:dg}


The discontinuous-Galerkin (DG) method dates back to 1973 when Reed \& Hill 
\cite{ReedHill73} developed the scheme for solving a neutron 
transport equation.
The theoretical framework for the DG method was solidified by 
Bernardo Cockburn and Chi-Wang Shu through a lengthy series of papers.
We refer the reader to their extensive review article and references
therein \cite{CoShu01}.
In this section, we define the notation used for the remainder of this 
paper and provide minimal details necessary for reproducing this body of work.
This section focuses on describing the Runge-Kutta discontinuous 
Galerkin (RKDG) scheme, and in \S\ref{subsec:dg-md}, we introduce the 
multiderivative technology to the DG framework.  We use similar notation to that
was previously introduced \cite{RoSe10}, and for further details on the DG method, 
we direct the reader the references (e.g. \cite{CoShu01,Seal12}).

Similar to the layout of \S\ref{sec:weno}, we begin this section
with the classical MOL formulation in \S\ref{subsec:dg-mol},
and continue in \S\ref{subsec:dg-md} with
the proposed multiderivative DG method.  We repeat that much like the
material from \S\ref{subsec:weno-md}, the multiderivative DG formulation relies on
the multiderivative ODE methods from \S\ref{sec:methods} but requires 
a very different application of the Cauchy
Kowalewski procedure from equation
\eqref{eqn:cauchy-kowalewski}.

\subsection{The discontinuous Galerkin (DG) method: a MOL formulation}
\label{subsec:dg-mol}

The DG method solves a discretization of the weak formulation of the 
hyperbolic conservation law \eqref{eqn:1dhyper}.
The continuous weak formulation can be realized by multiplying 
\eqref{eqn:1dhyper} with
a test function $\varphi$, and integrating by parts over a control volume
$\Tm = [x_\ell, x_r]$:
\begin{align} \label{eqn:1dweak-cont}
   \int_{x_\ell}^{x_r}\,  \varphi q_{,t}\, dx = 
         \int_{x_\ell}^{x_r} \varphi_{,x}\,  f(q)\, dx 
       -  \Bigl( 
           \left.\varphi f( q )\, \right\rvert_{x_r} - 
           \left.\varphi f( q )\,  \right\rvert_{x_\ell} \Bigr).
\end{align}

We begin our discretization by defining a grid containing $m_x$ cells 
for the domain $[a,b]$, each of whose width is $\Delta x = (b-a) / m_x$.
The $i^{th}$ grid cell is denoted by
$\mathcal{T}_i = [x_{i-1/2},x_{i+1/2}]$, where the cell edges are given by
\mbox{$x_{i-1/2} = a + (i-1)\dx$}, for $i=1,2,\ldots, m_x+1$, and the
cell centers are given by
\mbox{$x_{i} = a + (i-1/2)\dx$}, for $i=1,2,\ldots, m_x$.
For simplicity of exposition, we will restrict our attention to a uniform grid.
On this grid we define the {\it broken} finite element space
\begin{align}
   W^h = \left\{ w^h \in L^{\infty}(\Omega): \,
   w^h |_{\Tm} \in P^{p}, \, \forall \Tm \in \Tm_h \right\},
\end{align}
where $h = \Delta x$.
The above expression means that on each element $\Tm$, $w^h$ will
be a polynomial of degree at most $p$, and no continuity is assumed
across element edges.
Each element $\Tm_i$ can be mapped to the canonical element 
$\xi \in [-1,1]$ via the 
linear transformation 
\begin{equation}\label{eqn:canonical-transform}
   x = x_i + \xi \, \frac{\Delta x}{2}.
\end{equation}
%
Note that after a change of variables, spatial derivatives obey the following 
rule:
$\frac{\partial}{\partial x} = \frac{2}{\dx} \frac{\partial}{\partial \xi}$.
For the canonical element, we construct a set of basis functions that
are orthonormal with respect to the
following inner product:
\begin{align}
 \Bigl\langle \varphi^{(\ell)}, \, \varphi^{(k)} \Bigr\rangle :=  
 \frac1 2\int_{-1}^1 
        \varphi^{(\ell)}_{{}}(\xi) \varphi^{(k)}_{\text{}}(\xi)\ d\xi
   = \de_{\ell k},
\end{align}
where $ \de_{\ell k}$ is the Kronecker delta function.
This defines the Legendre basis functions:
\begin{align}
   \varphi^{(\ell)} = \Biggl\{ 1, \, \, \,  \sqrt{3} \, \xi,
    \, \, \,  \frac{\sqrt{5}}{2} \left( 3 \xi^2 - 1 \right), \, \, \,
   \frac{\sqrt{7}}{2} (5 \xi^3 - 3 \xi), \, \, \, \dots \Biggr\}.
\end{align}

We consider approximate solutions of the hyperbolic conservation law 
\eqref{eqn:1dhyper} that are defined by a finite set of coefficients
$\left\{ Q^{(k)}_i \right\}$ of the 
basis functions 
$\left\{ \varphi^{(k)}\right\}$.
When restricted to a single cell $\Tm_i$, the approximate solution $q^h$ is
\begin{equation}
\label{eqn:q_nsatz-1d}
 {q}^h(t, x ) \Bigl|_{\Tm_{i}}  :=
   q_i^h(t, \xi ) = 
     \sum_{k=1}^{M}  Q^{(k)}_{i}(t) \, \varphi_{\text{}}^{(k)}(\xi),
\end{equation}
where $M$ is the order of accuracy.
The initial conditions are 
determined from the $L^2$-projection of $q^h(0,x)$ onto the 
basis functions,
\begin{equation}
   \label{eqn:l2project1d}
   Q^{(k)}_{i}(0) := \Bigl\langle {q}_i^h(0, \xi), \,
                     \varphi_{\text{}}^{(k)}(\xi) \Bigr\rangle,
\end{equation}
which we evaluate using $M$ standard Gaussian quadrature points.

The semi-discrete weak formulation of \eqref{eqn:1dhyper} is given by 
discretizing \eqref{eqn:1dweak-cont} with a finite set of basis functions and
a finite set of control volumes.  Inserting our discrete spatial
representation \eqref{eqn:q_nsatz-1d} into the continuous 
weak formulation \eqref{eqn:1dweak-cont} and using
$\varphi = \varphi^{(k)}$ for our test function,
we produce the semi-discrete weak formulation
\begin{equation} \label{eqn:discrete-weak-long}
\begin{split}
   \Bigl\langle {q}_{i,t}^h( t, \xi ), \varphi^{(k)}(\xi) \Bigr\rangle
   =&\  \frac{2}{\dx} \Bigl\langle f\left( {q}_{i}^h( t, \xi ) \right),  
       {{\varphi_{\xi}^{(k)}}}{}(\xi) \Bigr\rangle \\
   & - \frac{1}{\dx} \varphi^{(k)}(\xi=+1) f^\downarrow( q^h(t,x_{i+1/2} ) \\
   & + \frac{1}{\dx} \varphi^{(k)}(\xi=-1) f^\downarrow( q^h(t, x_{i-1/2} ) ),
\end{split}
\end{equation}
where the \emph{flux values} $f^\downarrow( q^h(t, x_{i-1/2} ) )$ will be
defined through an appropriate Riemann solver.  These terms
will become responsible for all inter-cell communication.  

Before going into the details of the Riemann solver,
we seek a simpler representation of \eqref{eqn:discrete-weak-long}.
We start by defining the first term on the right hand side of
\eqref{eqn:discrete-weak-long} as the \emph{interior integral}
\begin{align}
   N^{(k)}_i := \frac{1}{\dx} \int_{-1}^1
   \varphi^{(k)}_{\xi}\left( \xi \right) f\left( q^h_i(t,\xi) \right) \, d\xi,
\end{align}
and after rescaling, we define the last two terms as
\begin{subequations}
\begin{align}
   {F}^{(k)}_{p, i+1/2} :=&\ \varphi^{(k)}(\xi=+1) f^\downarrow(q^h(t,x_{i+1/2})), \\
   {F}^{(k)}_{m, i-1/2} :=&\ \varphi^{(k)}(\xi=-1) f^\downarrow(q^h(t,x_{i-1/2})).
\end{align}
\end{subequations}
With these definitions in place, 
the discrete weak formulation can now be compactly written as a large MOL
system:
\begin{equation} \label{eqn:1dweak}
   \frac{d}{dt} {Q}^{(k)}_i(t) = 
   \underbrace{\phantom{\frac{1}{1}}{N^{(k)}_i (t)}}_\text{Interior}
   - 
   \underbrace{\frac{1}{\dx} \left[ {F}^{(k)}_{p,i+1/2} (t)  - {F}^{(k)}_{m,i-1/2}
     (t) \right]}_\text{Edge}.
\end{equation}
The only remaining piece to describe is how we compute the inter-cell flux
values, $f^\downarrow( q^h(t, x_{i-1/2} ) )$.

\subsubsection{The discontinuous Galerkin (DG) method: a choice of Riemann
solvers}

Given that the representation of the solution, $q^h$, is not forced to
be continuous at the cell interfaces, care must be taken when
defining the flux values 
$f^\downarrow\left( q^h(t,x_{i-1/2})\right)$.
One could work with the so-called \emph{generalized Riemann solvers} which
take into account spatially varying function to the left and right of the
discontinuity. One could also work with exact solutions for a classical 
\emph{Riemann problem} in which one assumes constant states to the left 
and right of a discontinuity.
In this work, we use \emph{Rusanov's method} \cite{Rusanov61},  
which is an approximate 
Riemann solver 
that is commonly
called the \emph{local Lax-Friedrichs} 
(LLF) Riemann solver given its similarity to the Lax-Friedrichs method.
Much like the Lax-Friedrichs method, this solver
approximates each Riemann problem with two waves with equal speeds
traveling in opposite directions, but LLF uses a local speed in 
place of a global speed to do so.

Evaluating the basis functions to the left and right of a single interface
located at $x_{i-1/2}$ provides artificially constant states, defined by a
finite sum,
\begin{subequations}
\begin{align}
\label{eqn:rp_state1}
Q_r &=  Q^{(1)}_{i} - \sqrt{3} \, Q^{(2)}_{{i}}+ \sqrt{5} \, Q^{(3)}_{{i}}
- \sqrt{7} \, Q^{(4)}_{{i}}  + \cdots \\
\label{eqn:rp_state2}
Q_{\ell} &=  Q^{(1)}_{{i-1}} + \sqrt{3} \, Q^{(2)}_{{i-1}} + \sqrt{5} \,
Q^{(3)}_{{i-1}} + \sqrt{7} \, Q^{(4)}_{{i-1}}  + \cdots.
\end{align}
\end{subequations}
The LLF solver defines a single value based on these two constant states:
\begin{align}
   f^\downarrow\left( Q_l, Q_r \right) := 
   \frac{1}{2} \Bigl[ \left( f\left(Q_l\right) + f\left( Q_r\right) \right) - \alpha \left( Q_r - Q_l \right) \Bigr].
\end{align}
We use a local value of $\alpha$, defined by
$\alpha = \max \left\{ \left|s^1\right|, \left| s^2 \right| \right\}$, where
$s^1$ and $s^2$ are the HLL(E) speeds \cite{HaLaLe83} defined by
\begin{subequations}
\begin{align}
   s^1 &= \min \left\{ \min_p \lambda^{(p)}\left( \hat{Q} \right),\,  
           \min_p \lambda^{(p)}\left( Q_l \right) \right\}, \\
   s^2 &= \max \left\{ \max_p \lambda^{(p)}\left( \hat{Q} \right),\,  
           \max_p \lambda^{(p)}\left( Q_r \right) \right\},
\end{align}
\end{subequations}
and $\lambda^{(p)}$ are the eigenvalues of $f'$.
For this work, we take the simple arithmetic average 
$\hat{Q} = \frac{1}{2}\left( Q_l + Q_r \right)$, but
we point out that
Roe averages could also be used \cite{Roe81}.

This completes the DG method of lines (MOL) discretization for our PDE.
The only remaining part is to evolve the discrete coefficients through
time.
This is usually performed by explicit, high-order Runge-Kutta methods, 
resulting in a Runge-Kutta discontinuous Galerkin (RKDG) method.  
In principle, one may potentially work with this discretization to form a
multiderivative integrator, which would require taking a second derivative of
\eqref{eqn:1dweak}.
However, difficulty will ensue when trying to compute the Jacobian
of the complex system of ODEs in \eqref{eqn:discrete-weak-long}.
Instead,
we turn towards the Lax-Wendroff/Cauchy-Kowalewski discontinuous Galerkin
methods to define a discrete second derivative.

\subsection{The discontinuous Galerkin (DG) method: a multiderivative formulation }
\label{subsec:dg-md}

In this section, we extend the work of Qiu, Dumbser and Shu 
\cite{QiuDumbserShu05} to accommodate multiderivative technology.
An investigation into other fluxes
would be an interesting topic of future study, including a comparison of 
various approximate Riemann solvers using multiderivative 
technology \cite{Qiu07}; moreover, an investigation into 
generalized Riemann solvers \cite{DuMu06,DuBaDiToMuDi08,MoCaDuTo12} may yield
some interesting results.
At present, our current goal is to lay the foundation that would be
necessary for such investigations.

Starting with the aim of defining a single stage value (or full update) 
for a multiderivative integrator,
a DG implementation of equation
\eqref{eqn:md-stage} from 
\S\ref{subsec:ode-integrators-example}
requires the definition of $q_{,tt}$ as well as
the definition of $q^\star_{,tt}$.  
Here, we use $q_{,tt} = \left( f'(q) f_x \right)_x$ in place of 
$q_{,tt} = -\left( f'(q) q_t \right)_x$ which was used in \S\ref{subsec:weno-md}
for the second derivative.
After factoring out a single spatial derivative, we can
write \eqref{eqn:md-stage} as 
%
%
\begin{equation}\label{eqn:dg-stage}
   q = q^n - \dt \Bigl[ \alpha f^n + \alpha^* f^* - 
       \dt \beta f'(q^n) f^n_x - \dt \beta^* f'(q^*) f^*_x \Bigr]_{x},
\end{equation}
%
%
where we have made use of the shorthand definitions
\begin{align}
   f^n := f(q^n), \quad f^* := f(q^*).
\end{align}
%
%
We remark that after defining the appropriate spatial discretization,
equation \eqref{eqn:dg-stage} will look strikingly similar to 
the semi-discrete weak formulation already presented in equation
\eqref{eqn:1dweak}.

Before proceeding onward, we argue that a complete understanding of how
to discretize
\eqref{eqn:dg-stage} will provide the necessary
building block for arbitrary 
multiderivative Runge-Kutta integrators.
For example, setting $\alpha^* = \beta^* =0$,
we can view this as a single-step method, and therefore
equation \eqref{eqn:dg-stage} is nothing other
than the Lax-Wendroff DG scheme already presented in the literature, provided
the correct coefficients are inserted.
Moreover, setting $\beta = \beta^* = 0$, this produces a two-stage Runge-Kutta
method.  
Further stages can be included by adding in additional terms
which introduces cumbersome notation.
For clarity, we restrict our attention towards defining the discrete 
formulation of \eqref{eqn:dg-stage} which is general enough to 
accommodate more complicated multiderivative time integrators.

The first step is to construct a Galerkin representation of the flux function
$f^h$ as well as necessary coefficients for the second time derivative,
$g^h := f'(q^h) f(q^h)_x$.
These two functions are the first and second, respectively, time derivatives
of $q$ before taking the final spatial derivative.
The Galerkin coefficients for these spatial derivatives can be constructed
by taking advantage of the basis functions:

%
%
\subsection*{\bf Procedure 1.1}

\smallskip

\noindent INPUT:
\begin{align*}
   \left\{Q^{(k)}_i\right\} \quad - \quad 
       \text{a list of Galerkin coefficients of $q^h$. \phantom{and}}
\end{align*}
OUTPUT: 
\begin{align*}
   \left\{F^{(k)}_i\right\} & \quad - \quad
       \text{a list of Galerkin coefficients of $f(q^h)$, and} \\
   \left\{{G}^{(k)}_i\right\} & \quad - \quad
       \text{a list of Galerkin coefficients of $f'(q^h) f_x$.}
\end{align*}
\begin{enumerate}

\item For each element $\Tm_i$, evaluate the flux function,
$f_i( \xi ) := f( q^h_i( \xi ) )$ and the Jacobian
$J_i( \xi ) := f'( q^h_i(\xi ) )$
at a list of $M$ quadrature points, 
$\left\{ \xi_1, \ldots \xi_M\right\}$.
Here, $\xi$ is the canonical variable for grid element $\Tm_i$
related to $x$ through \eqref{eqn:canonical-transform}.
%
\item 
Compute the projection of $f_i(\xi)$ onto the basis functions using the point 
values from Step 1:
\begin{align}
   F^{(k)}_i &:= \Bigl\langle f_i\left( \xi \right), \,
                 \varphi_{\text{}}^{(k)}(\xi) \Bigr\rangle.
\end{align}
This projection step produces a Galerkin expansion of the 
flux function on grid element $\Tm_i$, that when written in the canonical 
variable is given by,
\begin{align}
 {f_i}( \xi )  :=
     \sum_{k=1}^{M}  F^{(k)}_{i}(t) \, \varphi_{\text{}}^{(k)}(\xi).
\end{align}
\item
Differentiate the flux function $f_i(\xi)$
on the interior of cell $\Tm_i$ to produce 
\begin{align}
   {\partial_x{f_i}} \left( \xi \right) := 
   \frac{2}{\dx} \sum_{k=1}^M F^{(k)}_i 
       {\frac{\partial \varphi}{\partial \xi}}^{(k)} \left( \xi \right).
\end{align}
Note that this differentiation step is only valid on the interior of each
cell.
Evaluate this function at the interior quadrature points, 
$\left\{ \xi_1, \xi_2, \ldots, \xi_M \right\}$.
%
\item
Define the Galerkin coefficients of $g^h$ as,
\begin{align}
       G^{(k)}_i := 
\Bigl\langle  J_i \left(
\xi
\right) \cdot \partial_x{f_i} \left( \xi \right), \,
\varphi_{\text{}}^{(k)}(\xi) \Bigr\rangle.
\end{align}
The required point values for the integration for $J_i(\xi_m)$ were 
computed in Step 1, 
and the required point
values of $\partial_{x} f_i( \xi_m )$ were computed in Step 3 by differentiating
the basis functions.
\end{enumerate}
We remark that differentiating the basis functions loses a single order of
accuracy, but given that those terms involving $g^h$ are
multiplied by $\Delta t = \BigO(\dx)$, we recover the desired order of accuracy.  
We are now prepared to describe the full multistage multiderivative DG method.

\subsection*{\bf Multistage multiderivative DG procedure}
\smallskip

\begin{enumerate}
\item
Given a pair of Galerkin expansions $q^h$ and ${q^*}$, we compute four
Galerkin expansions ${f^h}, {g^h}, {f^*}$ and $g^*$ using 
Procedure 1.1.
\smallskip

\item
Define a modified flux function $\tilde{f}^h$ via
\begin{equation} \label{eqn:modified-flux}
   \tilde{f}^h := \alpha f^h + \alpha^* f^* + 
       \dt \left( \beta g^h + \beta^* g^* \right).
\end{equation}
We remark that we have a full Galerkin expansion of this modified flux 
function, 
that when restricted to a single element, is given by
\begin{gather}
   \left. \tilde{f}^h(x) 
   \right \rvert_{\Tm_i} =
   \sum_{k=1}^M 
\tilde{F}^{(k)}_i \varphi^{(k)}( \xi ), \\
\tilde{F}^{(k)}_i := \alpha F^{(k)}_i + \alpha^* {F^*}^{(k)}_i 
       + \dt \left( \beta G^{(k)}_i + \beta^* {G^*}^{(k)}_i \right).
\end{gather}
We can now simplify equation \eqref{eqn:dg-stage} by compactly writing it as
\begin{equation} \label{eqn:dg-md-tilde}
   q = q^n - \dt \tilde{f^h}_{x}.
\end{equation}

\smallskip

\item
Construct the time integrated weak formulation by 
multiplying \eqref{eqn:dg-md-tilde} by a test
function $\varphi^{(k)}$ and integrate by parts over a grid cell
$\Tm_i$:
\begin{align}
{Q}^{(k)}_i(t) &= 
{Q}^{(k)}_i(t^{n}) +
\dt \tilde{N}^{(k)}_i (t^n)
   - \frac{\dt}{\dx} \left[ 
           \tilde{F}^{(k)}_{p,i+1/2}
       (t^n)  
   - {\tilde{F}_{m,i-1/2}}^{(k)} (t^n) \right].
\end{align}
The only remaining piece is to define the flux values, as well as define
a proper left and right flux value for the Riemann solver.
The interior integrals $\tilde{N}^{(k)}_i$ are given by integration of the 
Galerkin expansion of $\tilde{f}^h$:
\begin{align}
   \tilde{N}^{(k)}_i := \frac{1}{\dx} \int_{-1}^1
   \varphi^{(k)}_{\xi} \tilde{f}^h\left( \xi\right) \, d\xi.
\end{align}
Once we have the Galerkin expansion of $\tilde{f}^h$, these integrals can be
evaluated exactly.
\smallskip

\item
Solve Riemann problems for the modified flux function.
In place of using equations \eqref{eqn:rp_state1} and \eqref{eqn:rp_state2}
to define the left and right hand side $F_\ell$ and $F_r$ of the Riemann 
problem,
we insert the extra time derivatives drawn from $\tilde{f}^h$ and tuck them into 
this evaluation.
That is, at the grid interface located at $x_{i-1/2}$, we redefine the left
and right values of $f^h$ to be the left hand side and right hand side evaluations
of $\tilde{f}^h$:
\begin{subequations}
\begin{align}
F_\ell :=& \hat{f}^h( x_{i-1/2}^- ) 
   = \sum_{k=1}^M \tilde{F}^{(k)}_{i-1} \varphi^{(k)}(\xi=+1)
   = \sum_{k=1}^M \sqrt{2k-1} \tilde{F}^{(k)}_{i-1}, \\
F_r :=& \hat{f}^h( x_{i-1/2}^+ ) 
   = \sum_{k=1}^M \tilde{F}^{(k)}_{i } \varphi^{(k)}(\xi=-1),
   = \sum_{k=1}^M (-1)^k \sqrt{2k-1} \tilde{F}^{(k)}_{i}.
\end{align}
\end{subequations}
\end{enumerate}

This completes the multiderivative description of the method within the 
discontinuous Galerkin framework.  Extensions to multiderivative integrators
that require more stages or more derivatives are a tedious, 
yet straight-forward extension of
what has already been presented in this section together with 
the methods presented in \S\ref{subsec:ode-integrators-example} for ODEs.
Extensions to methods with more stages
would require adding more values of $\alpha^*$ and
$\beta^*$ to the definition of the modified flux function,
 $\tilde{f}^h$ in equation \eqref{eqn:modified-flux}.  
For example,
defining the modified flux function by
\begin{equation}
   \tilde{f}^h := \left( \alpha f^h + \alpha^* f^* + \alpha^{**} f^{**} \right)
       + \dt \left( \beta g^h + \beta^* g^* + \beta^{**} g^{**} \right)
\end{equation}
allows us to use any ODE solver derived from \eqref{eqn:md-stage-td-3stage}
in \S\ref{subsec:ode-integrators-example}, such as the fifth-order
two-derivative method presented in \eqref{eqn:tdrk5}.

Extensions to methods with more derivatives
require bootstrapping previous Lax-Wendroff DG work 
\cite{QiuDumbserShu05}, much like
what has already been performed here.
This would involve
the appropriate extension of
Procedure 1.1
using the Cauchy-Kowalewski procedure from equation 
\eqref{eqn:cauchy-kowalewski},
which would define an appropriate modification
$\tilde{f}$ of the flux function $f$ that contains extra time information
of the PDE.

\subsection{The discontinuous Galerkin (DG) method: a choice of limiters}
\label{subsec:dg-limiters}

The one detail that has been left out of this discussion has been the choice
of limiting options.
It is a well known fact that high-order linear
methods exhibit oscillatory behavior
near discontinuities, and therefore to obtain physically relevant results,
one needs to choose a limiter that ideally retains high order accuracy
in smooth regions, and reduces to first order accuracy locally at the 
location of the shock.
In this work, we use the moment based limiter developed by Krivodonova 
\cite{Krivodonova07} for our numerical simulations, although other limiters
may certainly be used as well, with no change to the general framework presented
thus far.
This limiter is applied after each stage of the (multiderivative) 
Runge-Kutta method, and it modifies the higher order terms 
$\left\{Q^{(k)}_i, k \geq 2\right\}$ after each time step.
One advantage of using multiderivative methods is that expensive limiters need
to be applied less often when compared to high-order, single-derivative
counterparts because of the reduction in the number of stages.

\section{Numerical examples}
\label{sec:numerical_examples}
   
In this section, we present results of the proposed method on a variety 
of hyperbolic conservation laws.
In \S\S\ref{subsec:linadv} -- \ref{subsec:buckley}, we consider 
scalar examples: constant coefficient advection, 
and the Buckley-Leverett two-phase flow model.
In \S\S\ref{subsec:shallow} -- \ref{subsec:euler}, we 
present results for the shallow water and Euler equations.
All discontinuous Galerkin solutions were run using the
open-source software DoGPack \cite{dogpack}.

Unless otherwise noted, we use the following parameters:
\begin{itemize}
   \item The time integrator is the unique two-stage, two-derivative, fourth-order
   Runge-Kutta method (TDRK4) from equation \eqref{eqn:exp4}.
   \item All finite difference WENO simulations use the fifth order WENO5-Z
   reconstruction from \S\ref{sec:weno} and a constant CFL number of 
   $\nu = 0.4 = \max_{q^*} \left|f'(q^*)\right| \frac{\dt}{\dx}$.
   \item Every DG result is fourth-order accurate in space, 
   and uses a desired CFL number of $\nu = 0.08$.
   If the maximum allowable CFL number defined by
   $\nu_\text{max} = 0.085$ is violated, a smaller time step is chosen.
   For visualization, we plot exactly $4$ uniformly spaced points per grid
   cell.
\end{itemize}
For the two problems where other time integrators are compared
against the TDRK4 method, we use coefficients in 
Table \ref{table:dg-cfl} for the DG simulations.

\begin{table}
\normalsize
\centering
\caption{CFL parameters used for DG simulations that compare time integrators.
The SSP-RK3 method is the optimal third order SSP method developed by Shu and Osher
\cite{ShuOsher87}
and described by 
Gottlieb and Shu \cite{GoShu98}.  SSP-RK4 is 
a fourth-order, low-storage method with ten stages developed by 
Ketcheson \cite{Ke08}.
The maximum allowable CFL number, $\nu_\text{max}$ is near the maximum 
possible stable time step that each method permits for fourth-order spatial
accuracy.
We note that SSP-RK4 has a much higher CFL limit when compared to either
TDRK4 or SSP-RK3 because it incorporates many more stages, but
each stage requires expensive applications of limiters.
It would be interesting to investigate optimized versions of TDRK4 
using \eqref{eqn:md-stage} as a building block that
allow taking larger time steps without adding extra storage.
\label{table:dg-cfl}
}
 \begin{tabular}{|l||c|c|}
 \hline
 {\normalsize \bf \text{DG time stepping parameters} } & 
 {\normalsize \text{$\nu$ } } & 
 {\normalsize \text{$\nu_\text{max}$ } } \\
   \hline \hline
{\normalsize \text{Third-order (SSP-RK3) }    } & 0.125  & 0.130  \\
\hline
{\normalsize \text{Low-storage SSP Runge-Kutta (SSP-RK4) }    } & 0.44  & 0.45  \\
\hline
{\normalsize \text{Two-derivative Runge-Kutta (TDRK4) }  } & 0.08  & 0.085 \\
\hline
\end{tabular}

\end{table}

\subsection{Linear advection}
\label{subsec:linadv}

Our first examples are variations on the scalar linear advection equation with 
periodic boundary conditions:
\begin{equation}\label{eqn:advection}
   q_{,t} +  q_{,x} = 0, \quad x \in [-1,1].
\end{equation}

\subsubsection{Linear advection: a smooth example}

For a smooth example, we use initial conditions
\begin{equation}\label{eqn:advection.test_smooth.ICs}
   q(0,x) = q_0(x) = \sin(\pi x),
\end{equation}
and we run the simulation up to $t=2.0$,
at which point the exact solution is given by the initial conditions.
This is one of two problems where we compare popular time integrators
against the new method.  Convergence studies are presented in 
Tables \ref{table:linear_smooth-weno} and \ref{table:linear_smooth-dg}.
For the WENO scheme and for this problem only, 
we run this problem with a large CFL number of $\nu=0.9$ 
in order to make the temporal error the preponderant part of the total
error; 
accordingly, in Table 1 we observe the fourth-order time accuracy of 
RK4 and TDRK4.

Relative errors for the WENO scheme are defined by an
$L^2$ norm based on point-wise values:
\begin{align}
   \text{Error } := 
   \frac{ \sqrt{ \dx \sum_{i=1}^{m_x} \left( q^n_i - q(t^n, x_i) \right)^2 } }{ 
           \sqrt{ \dx \sum_{i=1}^{m_x} q(t^n, x_i) ^2  } }.
\end{align}
Errors for the DG simulations likewise use a relative $L^2$-norm, but this
can be based on integration against the exact solution using
moments of the solution (c.f. \cite{RoSe10}).

\begin{table}
\centering
\normalsize
\caption{Advection equation: smooth example.
Convergence analysis for the WENO scheme applied to 
the advection equation \eqref{eqn:advection} 
with periodic boundary conditions and smooth initial conditions 
\eqref{eqn:advection.test_smooth.ICs}.
The table shows a comparison of the relative $L^2$-norm of the errors in the 
numerical solutions obtained with WENO5-Z spatial discretization and
different time integrators: 3rd-order strong stability preserving Runge-Kutta 
(SSP-RK3) \cite{ShuOsher87}, classical 4th-order Runge-Kutta (RK4), and the
new two-derivative 4th-order 
Runge-Kutta (TDRK4) method.
The Courant parameter chosen for this problem was a constant 
CFL of $\nu = 0.9 = { \Delta t } / { \Delta x}$.
In order to observe fourth-order accuracy for the fourth-order methods,
we needed to increase the CFL number and therefore increase the temporal error.
Results for smaller CFL numbers reached machine precision before announcing
their accuracy, where
both RK4 and TDRK4, indicated convergence orders between 
$4^{th}$ and $5^{th}$ order.
%
The `Order' columns refer to the algebraic order of convergence, computed as 
the base-2 logarithm of the ratio of two successive error norms.
We remark that for all CFL numbers and weighting schemes tested,
including WENO-Z, WENO-JS and linear weights, 
both time integrators RK4 and TDRK4 have comparable errors.
\label{table:linear_smooth-weno}}
\begin{tabular}{|r||c|c||c|c||c|c|}
\hline
\bf{Mesh} & \bf{SSP-RK3} & \bf{Order} & \bf{RK4} & \bf{Order} & \bf{TDRK4} & \bf{Order}\\
\hline
\hline
$  25$ & $3.09\times 10^{-03}$ & --- & $2.00\times 10^{-04}$ & --- & $1.52\times 10^{-04}$ & ---\\
\hline
$  50$ & $3.79\times 10^{-04}$ & $3.03$ & $9.68\times 10^{-06}$ & $4.37$ & $8.77\times 10^{-06}$ & $4.11$\\
\hline
$ 100$ & $4.74\times 10^{-05}$ & $3.00$ & $5.54\times 10^{-07}$ & $4.13$ & $5.39\times 10^{-07}$ & $4.02$\\
\hline
$ 200$ & $5.92\times 10^{-06}$ & $3.00$ & $3.37\times 10^{-08}$ & $4.04$ & $3.35\times 10^{-08}$ & $4.01$\\
\hline
$ 400$ & $7.39\times 10^{-07}$ & $3.00$ & $2.09\times 10^{-09}$ & $4.01$ & $2.09\times 10^{-09}$ & $4.00$\\
\hline
$ 800$ & $9.24\times 10^{-08}$ & $3.00$ & $1.31\times 10^{-10}$ & $4.00$ & $1.31\times 10^{-10}$ & $4.00$\\
\hline
$1600$ & $1.16\times 10^{-08}$ & $3.00$ & $8.33\times 10^{-12}$ & $3.97$ & $8.33\times 10^{-12}$ & $3.97$\\
\hline
\end{tabular}

\end{table}

\begin{table}
\centering
\normalsize
\caption{Advection equation: smooth example.
We present DG-results comparing the new multiderivative scheme TDRK4 against 
third-order SSP-RK3 \cite{ShuOsher87,GoShu98} and a state of the art
low-storage fourth-order Runge-Kutta method (SSP-RK4) \cite{Ke08}.
The CFL numbers chosen for each scheme are presented in Table \ref{table:dg-cfl},
which are near the maximum stable CFL limit for each scheme.
\label{table:linear_smooth-dg}}
\begin{tabular}{|r||c|c||c|c||c|c|}
\hline
\bf{Mesh} & \bf{SSP-RK3} & \bf{Order} & \bf{SSP-RK4} & \bf{Order} & \bf{TDRK4} & \bf{Order}\\
\hline
\hline
$   5$ & $1.17\times 10^{-03}$ & --- & $6.18\times 10^{-04}$ & --- & $8.21\times 10^{-04}$ & ---\\
\hline
$  10$ & $1.32\times 10^{-04}$ & $3.16$ & $3.87\times 10^{-05}$ & $4.00$ & $5.99\times 10^{-05}$ & $3.78$\\
\hline
$  20$ & $1.60\times 10^{-05}$ & $3.04$ & $2.43\times 10^{-06}$ & $3.99$ & $3.91\times 10^{-06}$ & $3.94$\\
\hline
$  40$ & $1.99\times 10^{-06}$ & $3.01$ & $1.52\times 10^{-07}$ & $4.00$ & $2.47\times 10^{-07}$ & $3.98$\\
\hline
$  80$ & $2.48\times 10^{-07}$ & $3.00$ & $9.51\times 10^{-09}$ & $4.00$ & $1.55\times 10^{-08}$ & $4.00$\\
\hline
$ 160$ & $3.10\times 10^{-08}$ & $3.00$ & $5.94\times 10^{-10}$ & $4.00$ & $9.69\times 10^{-10}$ & $4.00$\\
\hline
$ 320$ & $3.87\times 10^{-09}$ & $3.00$ & $3.72\times 10^{-11}$ & $4.00$ & $6.03\times 10^{-11}$ & $4.01$\\
\hline
$ 640$ & $4.84\times 10^{-10}$ & $3.00$ & $2.24\times 10^{-12}$ & $4.05$ & $4.39\times 10^{-12}$ & $3.78$\\
\hline
\end{tabular}

\end{table}

\subsubsection{Linear advection: a challenging example}
The second example we test is more challenging.
The initial conditions are a 
combination of four bell-like shapes with different degrees of smoothness, 
that was originally proposed by Jiang and Shu \cite{JiangShu96}:
\begin{subequations}
\begin{align}
\label{eqn:square_wave_ic-a}
 &q_0(x) = 
 \begin{cases}
   \frac{1}{6} \left[
   G(x,\beta,z-\delta) + G(x,\beta,z+\delta) + 4 G(x,\beta,z) \right],
                  & \quad -0.8 \le x \le -0.6; \\
   1,             & \quad -0.4 \le x \le -0.2; \\
   1-|10(x-0.1)|, & \quad  0   \le x \le  0.2; \\
   \frac{1}{6} \left[
   F(x,\alpha,a-\delta) + F(x,\alpha,a+\delta) + 4 F(x,\alpha,a) \right],
                  & \quad  0.4 \le x \le  0.6; \\
   0,             & \quad \text{otherwise}.
 \end{cases}\\
 \label{eqn:square_wave_ic-b}
 &G(x,\beta ,z) = e^{-\beta(x-z)^2};\\
   \label{eqn:square_wave_ic-c}
 &F(x,\alpha,a) = \sqrt{\max(1-\alpha^2(x-a)^2,0)}.
\end{align}
\end{subequations}
%
The constants are $a=0.5$, $z=-0.7$, $\delta=0.005$, $\alpha=10$ and 
$\beta = \log_{10}(2)/(36\,\delta^2)$.
With this choice of constants, the Gaussian curve centered
at $x=-0.7$ contains two discontinuities, and the bump
centered at $x=0.5$ also has two discontinuities.
Moreover, this last bump has no derivative at $x=0.405$ and
$x=0.595$.

After testing many different reconstruction procedures of varying orders,
many different time integrators and a large range of parameters,
our observations indicate that all WENO methods are sensitive to the 
parameters used for this test problem.
It has already been observed that WENO schemes may require tuning
in order to achieve good results \cite{BoCaCoWai08};
this is not the subject of this work, but we find it necessary to point
out that WENO simulations do not emulate robust behavior for this problem.
A numerical investigation into their behavior on linear problems with difficult
initial conditions such as \eqref{eqn:square_wave_ic-a}-\eqref{eqn:square_wave_ic-c}
would be interesting.

Results for this problem are presented in Figure \ref{fig:advection}.
The under-resolved DG simulations tend to be quite diffusive when 
compared to the WENO schemes, but behave much more predictably.

\begin{figure}[!htb]
\centering
\includegraphics[width=\textwidth]{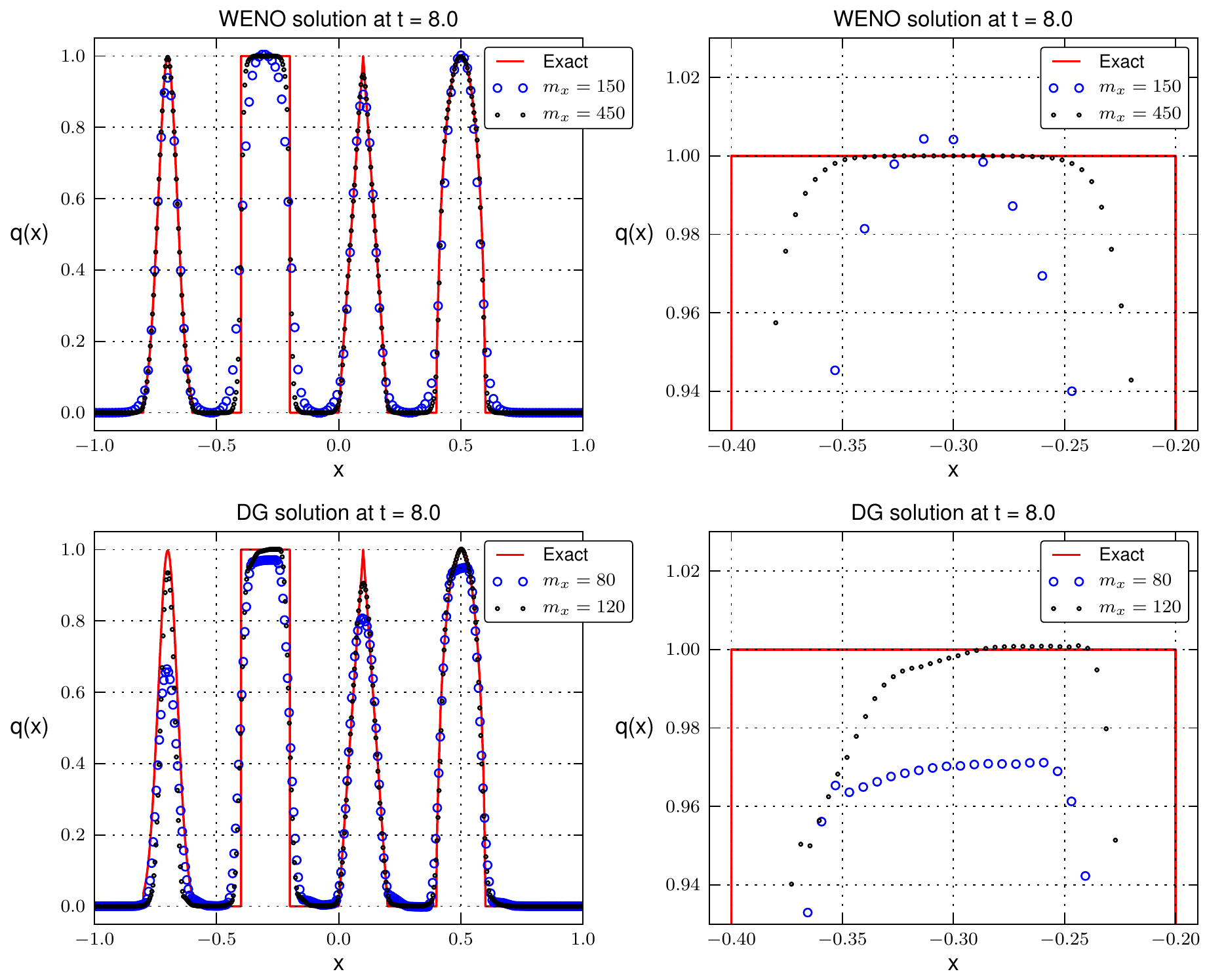}
\caption{Advection equation: a discontinuous example.
Results for the two-derivative (TDRK4) method for WENO and DG spatial 
discretizations.
Simulations are run to a final time of $t = 8.0$ for a total of four full
revolutions.
The two images on the right are zoomed in images of the top of the square
wave.
WENO simulations use a CFL number of $\nu = 0.4$, and the DG simulations
use a constant CFL number of $\nu = 0.08$.
We again plot four points per cell for the DG solutions.
\label{fig:advection}
}
\end{figure}

\subsection{Buckley-Leverett}
\label{subsec:buckley}

The Buckley-Leverett equation is a non-linear scalar problem with a non-convex
flux function.
Years past its invention, 
it has become a standard benchmark problem for many hyperbolic
solvers. 
The model describes two-phase flow through 
porous media, where the application is oil recovery from a reservoir 
containing an oil-water or oil-gas mixture of fluids \cite{BuckLev42}.
In rescaled units, the flux function for this problem is defined through a 
single free parameter $M$ with
\begin{align}
   f(q) = \frac{q^2}{q^2+M(1-q)^2}.
\end{align}
We use $M = 1/3$, and we take the computational domain to be
$[-1,1]$ with initial conditions prescribed through
\begin{align}
   q(0,x) = 
   \begin{cases}
   1, \quad \text{if } -1/2 \leq x \leq 0, \\
   0, \quad \text{otherwise}.
   \end{cases}
\end{align}

For small time values, the exact solution to this problem can be found by 
solving two Riemann problems,
where each pair of states 
produces a typical `compound wave'
that is the combination of a leading shock and a trailing rarefaction.
Each shock propagates with a speed dictated by
the Rankine-Hugoniot conditions
\begin{gather}
   f'( q^* ) = \frac{ f(q^* ) - f( q_c ) }{ q^* - q_c },
\end{gather}
where $q_c$ is the constant value that does not change.
For the Riemann problem located at $x=-0.5$, we have 
$q_c =1$, and  $q^* \approx 0.1339745962155613$, and for the
Riemann problem located at $x=0.0$, we have $q_c = 0$ and
$q^* = 0.5$.
Characteristics between $q^*$ and $q_c$ propagate with speed $f'(q)$,
and therefore to fill the rarefaction fan, we plot a range of values
$( s + t f'(q), q )$, where $q$ is a sampling of values between $q_c$ and 
$q^*$, and $s \in \left\{-0.5,0.0 \right\}$ is the location of each
Riemann problem.
Results for the two spatial discretizations 
are presented in Figures 
\ref{fig:buckley-leverett-weno} and
\ref{fig:buckley-leverett-dg}
at a final time of $t=0.4$.
In addition, we compare three different two-derivative methods 
that were presented
in \S\ref{subsec:ode-integrators-example} in 
Figure \ref{fig:buckley-leverett-weno-tdrk-comp}.

As a final note, we remark that for this problem only, we modify our scheme
to deal with some pathological issues.
Given that the flux function is non-convex, the WENO simulations
use the analytical value for $\alpha$, which is approximately
\mbox{$\alpha := \max_{q\in [0,1]} |f'(q)| \approx 2.205737062$}.
For the DG simulations, we use the HLL(E) Riemann solver \cite{HaLaLe83}.
In this case, the HLL(E) solver performs better than the local Lax-Friedrichs
(LLF) solver given that $f'(q) \geq 0$ for all $q$ in our domain. 
The LLF solver approximates this problem with
two waves traveling in opposite directions, whereas the exact solution 
travels in one direction only, which the HLL(E) correctly accounts for.

\begin{figure}[!htb]
\centering
\includegraphics[width=\textwidth]{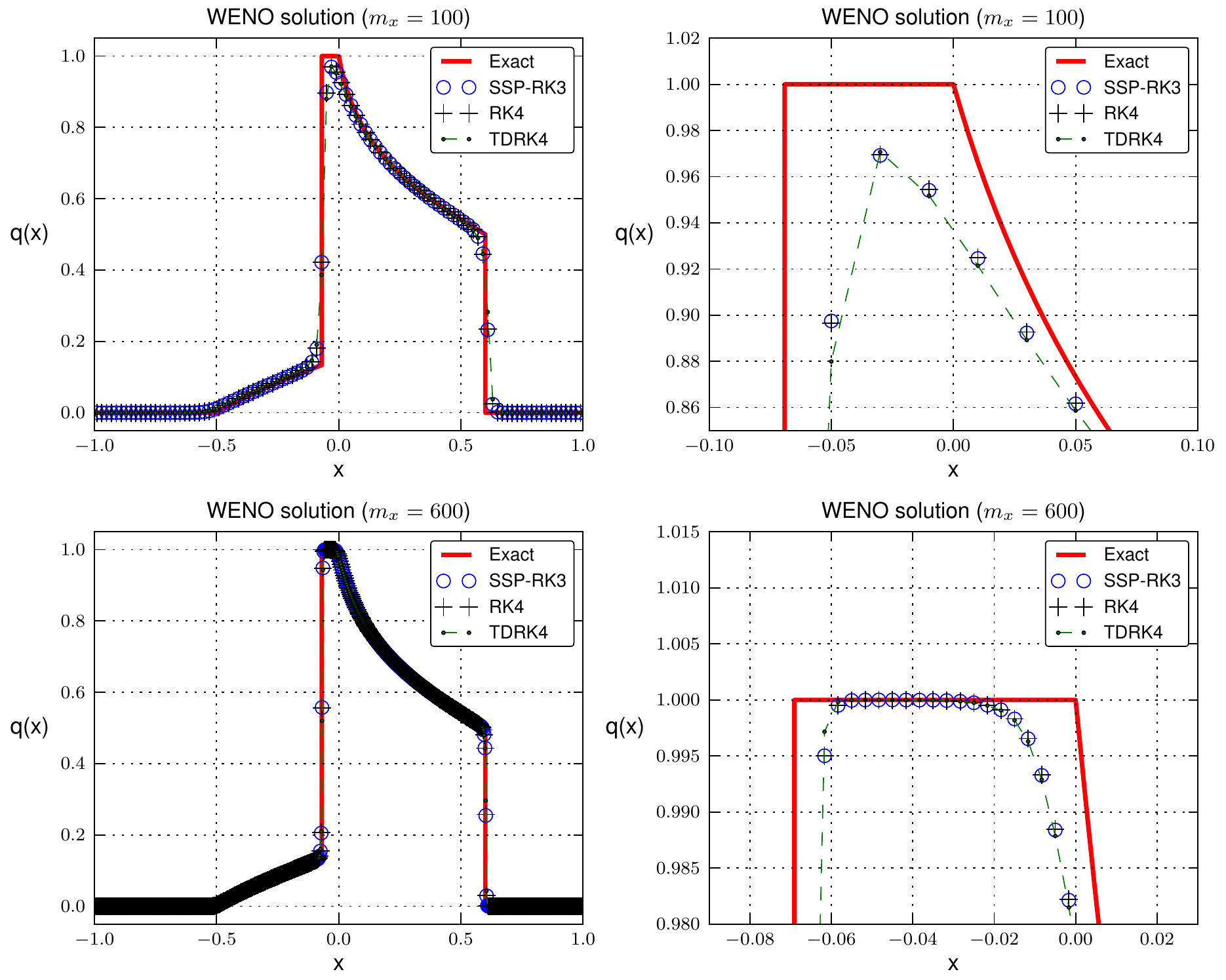}
\caption{Buckley-Leverett double Riemann problem: WENO solutions.
WENO results at the final time of $t=0.4$.
All time integrators compared for this method are giving similar results.
We remark that for this problem, the WENO schemes tend to be more diffusive
where the rarefactions form when compared to the DG schemes.
}
\label{fig:buckley-leverett-weno}
\end{figure}

\begin{figure}[!htb]
\centering
\includegraphics[width=\textwidth]{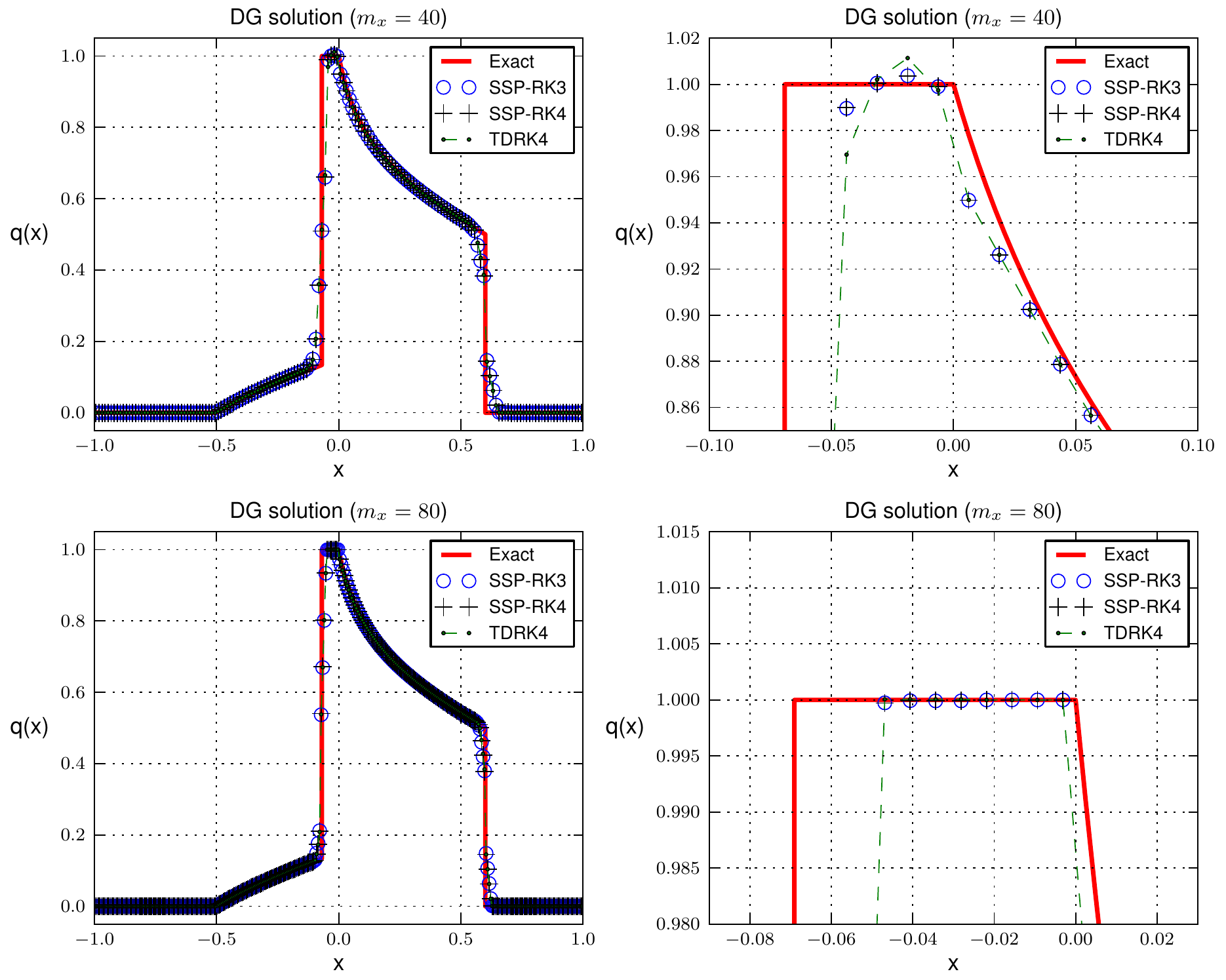}
\caption{Buckley-Leverett double Riemann problem: DG solutions.
DG results at the final time of $t=0.4$.
We plot exactly four uniformly spaced grid points per cell.
The top row uses $m_x = 40$ grid cells, and the bottom row uses
$m_x = 80$ grid cells.
The two frames on the right are zoomed in images of the solution near $x=0.0$, 
one of the few places where we were able to find
differences in the solutions.
We note that the coarse solution is under resolved because it 
only has a single cell to capture 
the structure at the top of the solution.
CFL numbers are chosen as in Table \ref{table:dg-cfl}, which
are close to the maximum stable time step allowed.
All time integrators are qualitatively giving the same result.
}
\label{fig:buckley-leverett-dg}
\end{figure}

\begin{figure}[!htb]
\centering
\includegraphics[width=\textwidth]{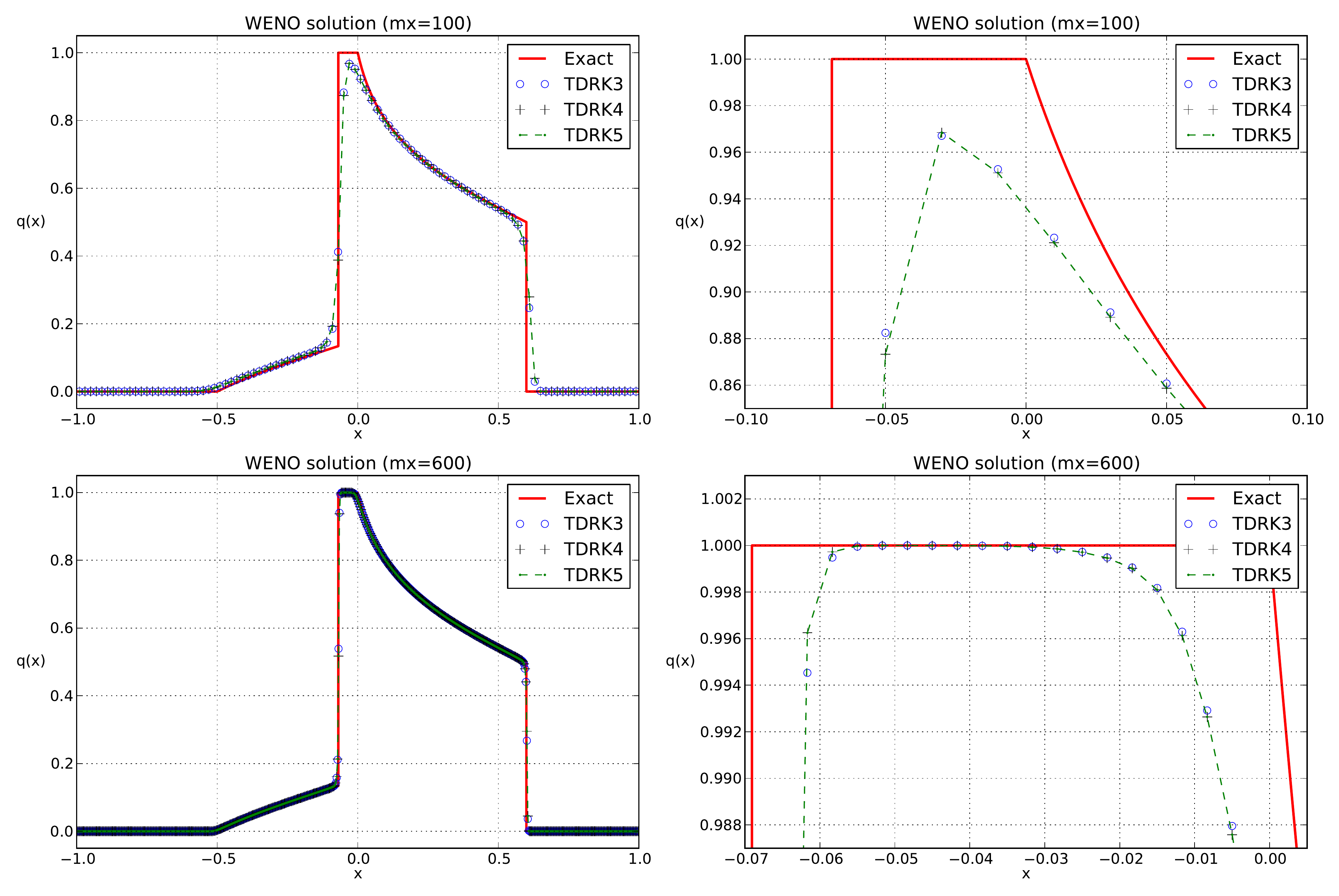}
\caption{Buckley-Leverett double Riemann problem: WENO solutions.
WENO results at the final time of $t=0.4$.
Here, we compare three different two-derivative methods that are presented
in \S\ref{subsec:ode-integrators-example}:
TDRK3 \eqref{eqn:tdrk3}, TDRK4 \eqref{eqn:exp4} and TDRK5 \eqref{eqn:tdrk5}.
The only noticeable difference is given by the third-order scheme,
and we attribute this to the lower order temporal accuracy.
}
\label{fig:buckley-leverett-weno-tdrk-comp}
\end{figure}

\subsection{Shallow water}
\label{subsec:shallow}

The shallow water equations define a hyperbolic system with two conserved
quantities: the water height $h$, and velocity $u$.
The system is defined by
\begin{align}
\left( 
   \begin{array}{c}
       h \\ hu 
   \end{array}
\right)_{, t}
+ 
\left( 
   \begin{array}{c}
       hu \\ h u^2 + \frac{1}{2} g h^2
   \end{array}
\right)_{, x}
= 0.
\end{align}
For our simulation, we take $g=1$.
We demonstrate our method on the dam break Riemann problem \cite{bLe02} 
with initial conditions defined by
\begin{align}
   (h, u )^T = 
   \begin{cases}
   (3,0)^T, \quad \text{if }  x \leq 0.5, \\
   (1,0)^T, \quad \text{otherwise}.
   \end{cases}
\end{align}
We use a computational domain of $[0,1]$ with outflow boundary conditions,
and stop the simulation at a final time of $t = 0.2$.
Results for the two-derivative time integrators are presented in Figure
\ref{fig:shallow-dam-weno} for the WENO method, and in 
Figure \ref{fig:shallow-dam-dg} 
for the discontinuous Galerkin scheme.
Exact solutions for this problem can be found in textbooks \cite{bLe02}.

\begin{figure}[!htb]
\centering
\includegraphics[width=\textwidth]{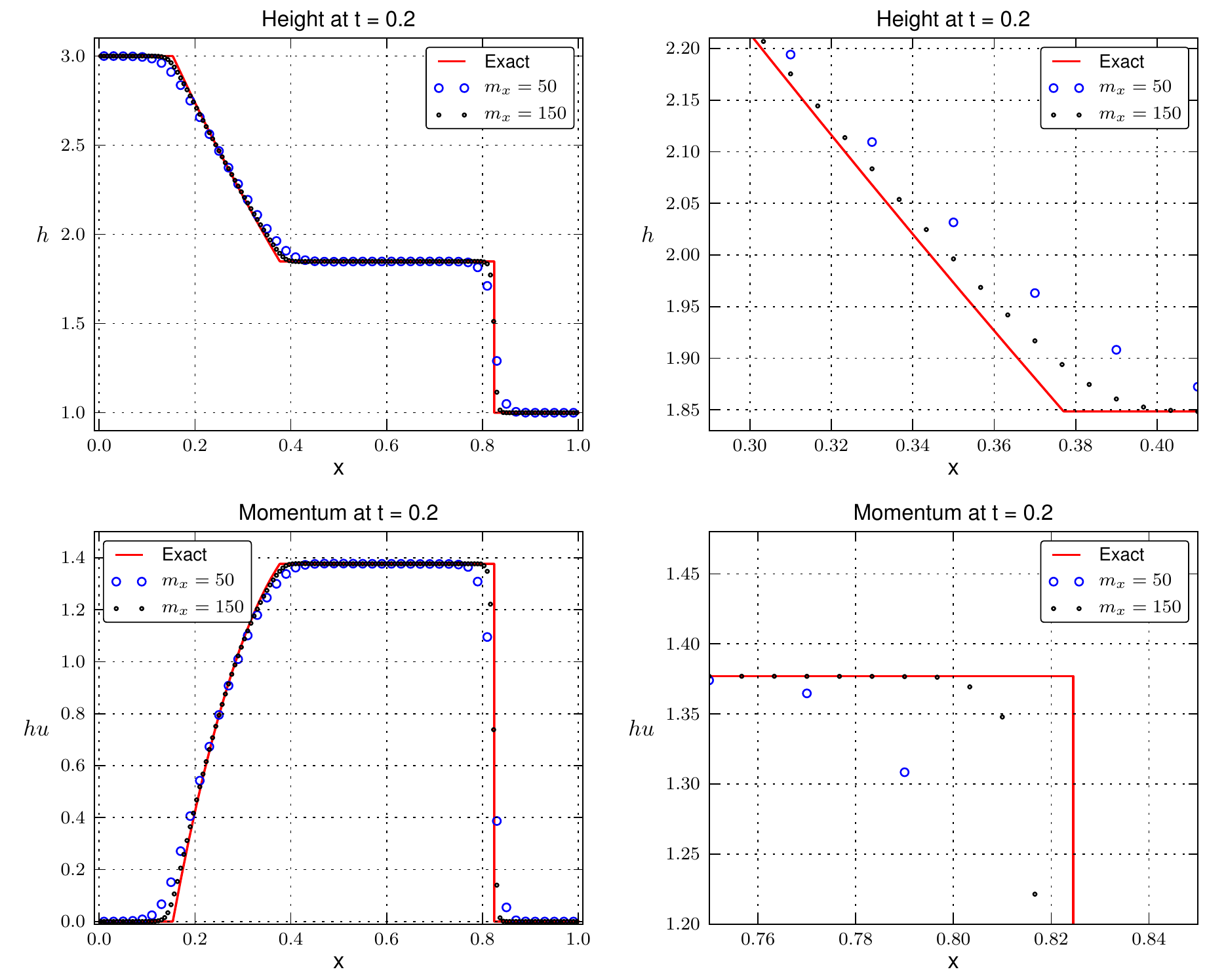}
\caption{Shallow water Riemann problem: WENO solutions.
Here we present two resolutions on top of each other using the new
multiderivative scheme presented in \S\ref{subsec:weno-md}.
We use $m_x = 50$ points for a coarse resolution and $m_x=150$ points
for a finer solution.
Top row: water height $h$ at the final time
with a zoom in of the rarefaction to the right.
Bottom row: momentum $hu$, with a zoom in of the shock to the right.
}
\label{fig:shallow-dam-weno}
\end{figure}

\begin{figure}[!htb]
\centering
\includegraphics[width=\textwidth]{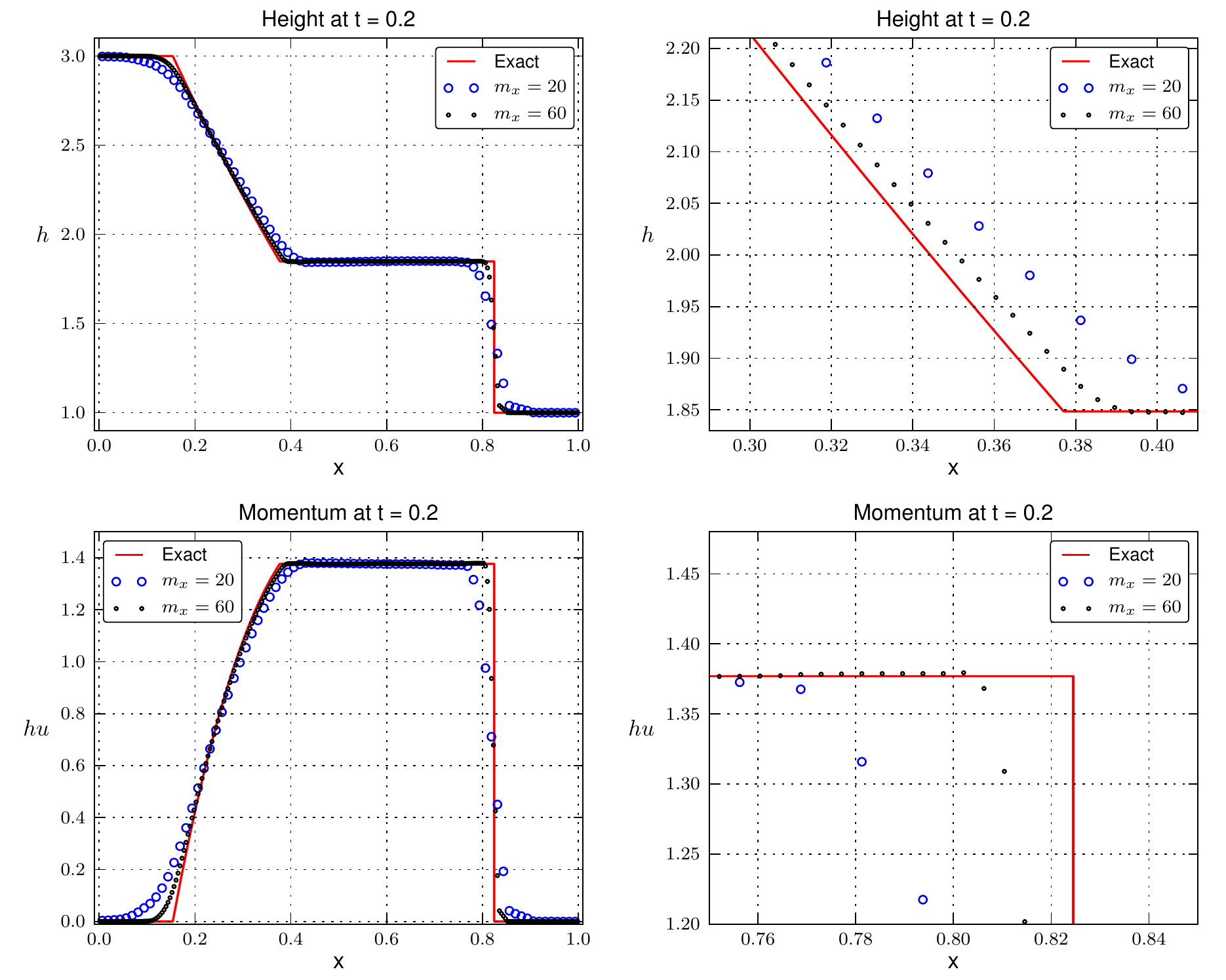}
\caption{Shallow water Riemann problem: DG solutions.
Here we present two resolutions on top of each other using the new
multiderivative scheme presented in \S\ref{subsec:dg-md}.
We use $m_x = 20$ cells for a coarse resolution and $m_x=60$ cells for a
finer solution, and we plot exactly four uniformly spaced points per
cell.
Top row: water height $h$ at the final time
with a zoom in of the rarefaction to the right.
Bottom row: momentum $hu$, with a zoom in of the shock to the right.
}
\label{fig:shallow-dam-dg}
\end{figure}

\subsection{Euler equations}
\label{subsec:euler}
The Euler equations describe the evolution of density $\rho$, momentum 
$\rho u$ and energy $\E$ of an ideal gas:
\begin{align}
\left( 
\begin{array}{c}
\rho \\ \rho u \\ \E
\end{array}
\right)_{, t}
+ 
\left( 
\begin{array}{c}
   \rho u \\ \rho u^2 + p \\ ( \E + p ) u
\end{array}
\right)_{, x}
= 0,
\end{align}
where $p$ is the pressure.
The energy $\E$ is related to the primitive variables $\rho$, $u$ and $p$ by
\begin{align}
   \E = \frac{p}{ \gamma-1 } + \frac{1}{2}\rho u^2,
\end{align}
where $\gamma$ is the ratio of specific heats.  For all of our simulations,
we take $\gamma = 1.4$.

%

\subsubsection{Euler equations: a shock tube Riemann problem}


We present a classic test case of a difficult 
shock tube, which is commonly referred
to as the \emph{Lax shock tube}.
The initial conditions are those defined by Harten 
\cite{Harten78,HaEnOsCh87}:
\begin{align}
\label{eqn:shock-harten}
   (\rho, \rho u, \E )^T = 
   \begin{cases}
   (0.445, 0.3111, 8.928)^T, \quad \text{if } x \leq 0.5, \\
   (0.5, 0, 1.4275)^T, \quad \text{otherwise}.
   \end{cases}
\end{align}
Exact solutions for this problem are well
understood, and there are many textbooks
that describe how to construct them \cite{bLe02,Toro99}.
For this set of data, the solution contains 
a left rarefaction, a contact
discontinuity and a shock wave traveling to the right.

We select $t = 0.16$ for the final time of our simulation 
\cite{QiuShu03}, and we use a computational domain of $[0,1]$
with outflow boundary conditions.
Results for this problem are presented in Figures 
\ref{fig:euler-shock-tube-harten-weno}
and 
\ref{fig:euler-shock-tube-harten-dg}.
For the WENO simulations, we additionally compare the two-derivative
two-derivative time integrator TDRK4 against the SSP-RK3
with two different CFL numbers, $\nu = 0.01$ and then $\nu = 0.4$.
We present results for the time integrator comparison in Figure
\ref{fig:euler-shock-tube-harten-compare-weno}.

\begin{figure}[!htb]
\centering
\includegraphics[width=\textwidth]{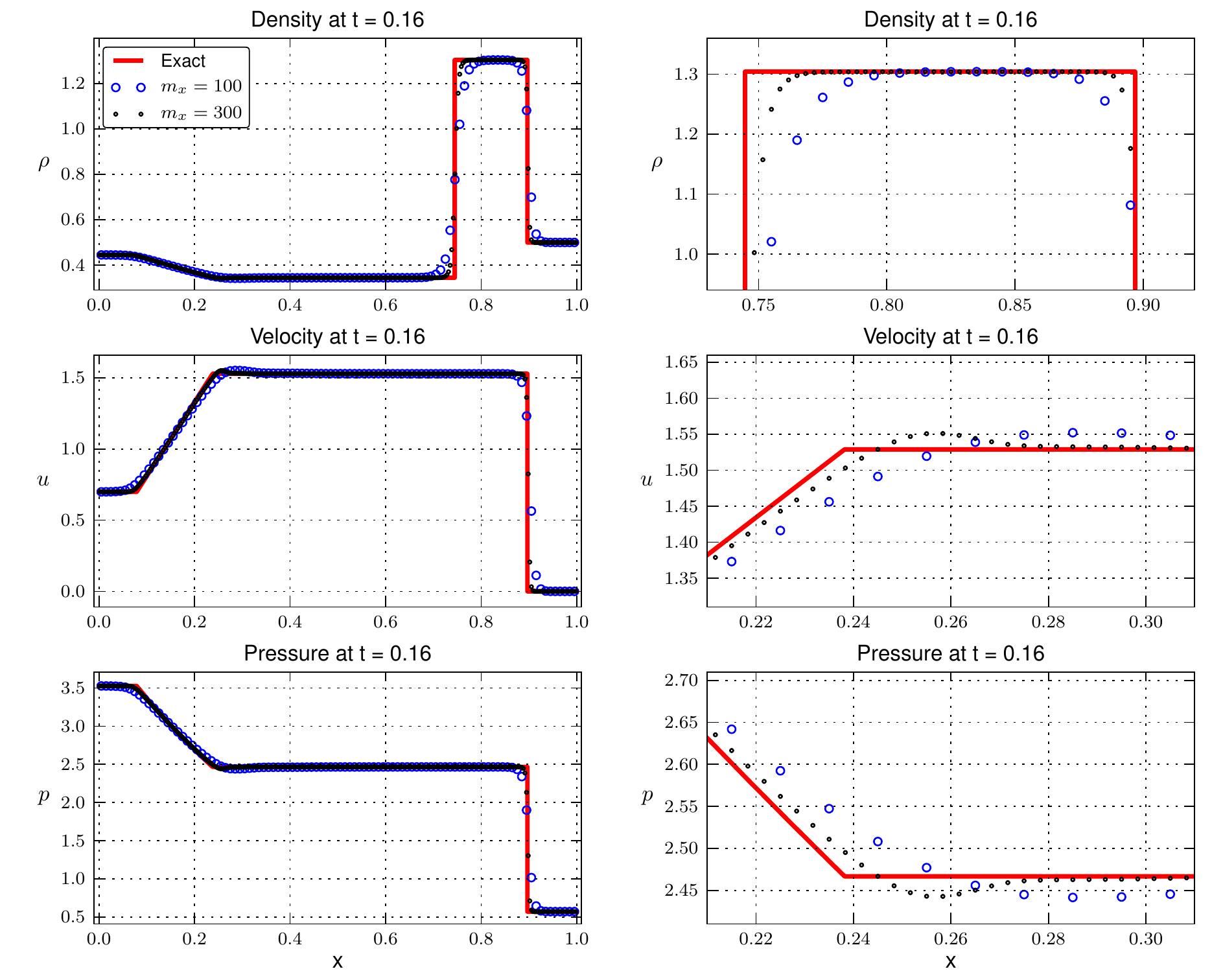}
\caption{Shock-tube Riemann problem: WENO solutions.
Shown here are WENO simulations with Harten's initial conditions  \eqref{eqn:shock-harten}
for the shock tube problem.
We plot observable quantities from top to bottom: 
density $\rho$, 
velocity $u$,
and 
pressure $p$.
Left columns are the full solution, and right columns are zoomed in parts
of the same data points.
We present a coarse solution with $m_x = 100$ points and a finer solution
of $m_x = 300$ points.
The zoomed in image of the density indicates the top part of the square section
between the contact discontinuity and the right traveling shock.
The zoomed in images for the velocity and pressure focus in on the 
right foot of the rarefaction fan.
}
\label{fig:euler-shock-tube-harten-weno}
\end{figure}

\begin{figure}[!htb]
\centering
\includegraphics[width=\textwidth]{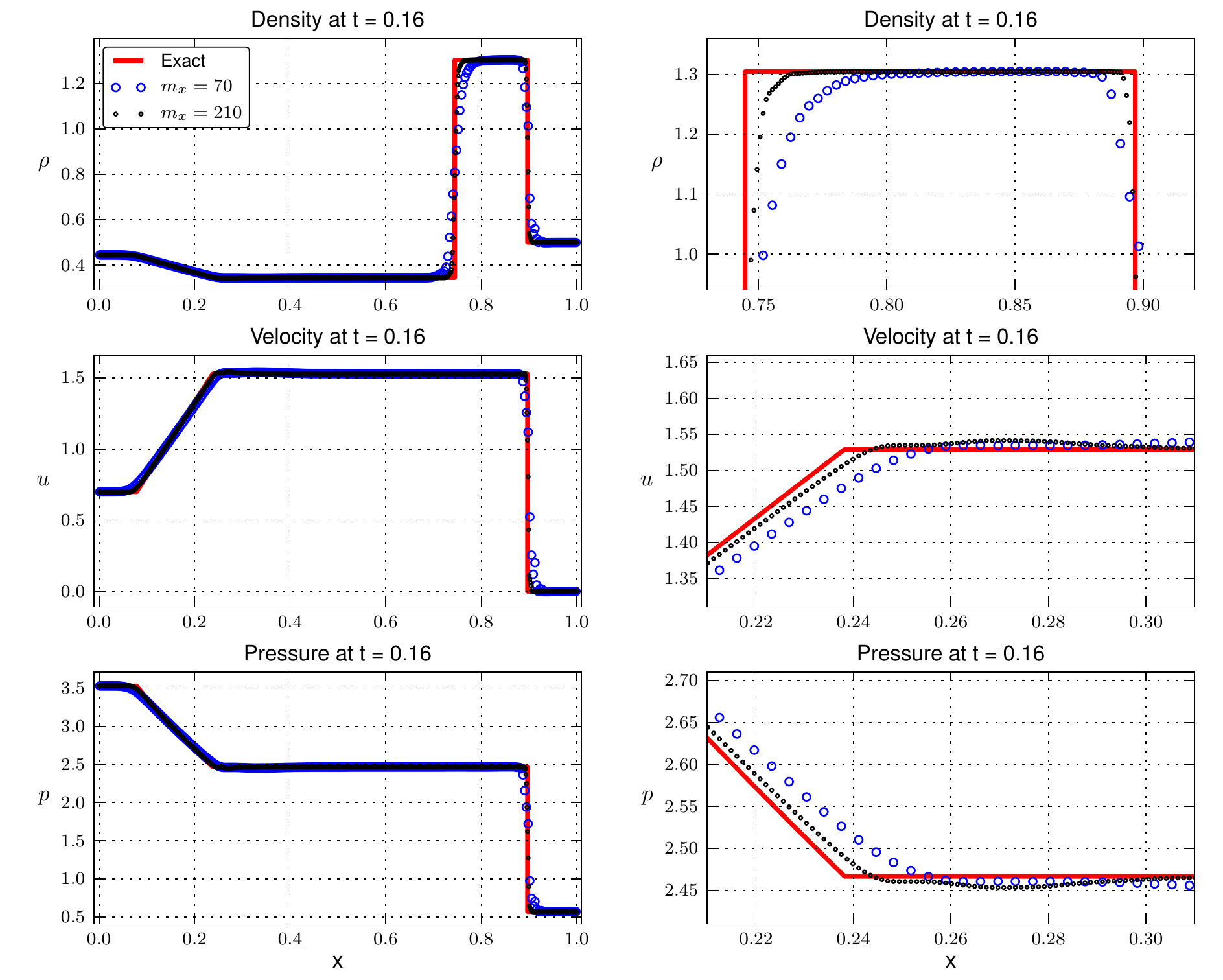}
\caption{Shock-tube Riemann problem: DG solutions.
We present two DG simulations for the physical
observables, which from top to bottom
are:
density $\rho$, 
velocity $u$,
and 
pressure $p$.
Left columns are the full solution, and right columns are zoomed in parts
of the same data points.
We present a coarse solution with $m_x = 70$ grid cells and a finer solution
of $m_x = 210$ grid cells, and we plot four uniformly spaced points per cell.
The axes are identical to those in
Figure 
\ref{fig:euler-shock-tube-harten-weno}.
}
\label{fig:euler-shock-tube-harten-dg}
\end{figure}

\begin{figure}[!htb]
\centering
\includegraphics[width=\textwidth]{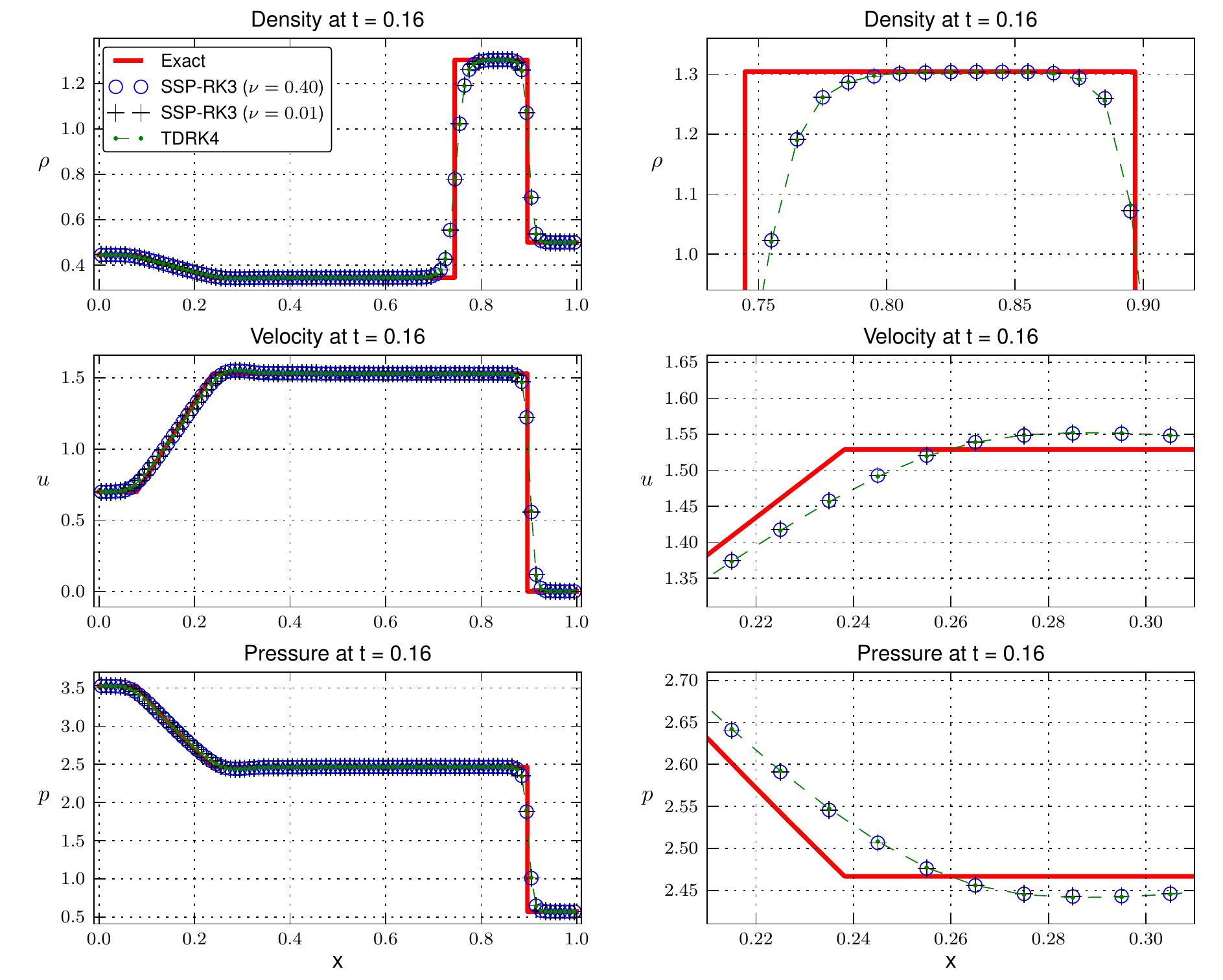}
\caption{Shock-tube Riemann problem: WENO solutions.
Shown here are WENO simulations comparing two different time integrators:
the third-order SSP method of Osher and Shu against the new two-stage, 
two-derivative method (TDRK4).
We repeat that all WENO simulations with the TDRK4 time integrator use a CFL
number of $\nu = 0.4$.
Here, we observe that the third-order SSP method with CFL numbers of 
$\nu = 0.4$ and $\nu = 0.01$ behaves qualitatively the same as the TDRK4 method.
We plot observable quantities from top to bottom: 
density $\rho$, 
velocity $u$,
and 
pressure $p$.
Left columns are the full solution, and right columns are zoomed in parts
of the same data points.
The resolution for this picture is the coarse solution of $m_x = 100$
identical to the coarse solution used in 
Figure \ref{fig:euler-shock-tube-harten-weno}.
Again, we remark that the accepted time integrator is in close agreement 
with the proposed method.
}
\label{fig:euler-shock-tube-harten-compare-weno}
\end{figure}

\subsubsection{Euler equations: shock entropy}
\label{subsubsec:shock-entropy}

Our final test case is another problem that is popular in the 
literature \cite{ShuOsher89}.  The initial conditions are
\begin{align*}
   (\rho, u, p) =& \left( 3.857143, 2.629369, 10.3333 \right), \quad &x < -4, \\
   (\rho, u, p) =& \left( 1 + \epsilon \sin(5x), 0, 1 \right), \quad &x \geq -4,
\end{align*}
with a  computational domain of $[-5,5]$.  
The final time for this simulation is $t=1.8$.
With $\eps = 0$, this is a pure Mach 3 shock moving to the right.
We follow the common practice of setting $\eps = 0.2$.

Results for WENO simulations are presented in Figure \ref{fig:euler-shock-entropy-weno}, and DG results are presented
in Figure \ref{fig:euler-shock-entropy-dg}.
For a reference solution, we plot a WENO simulation that uses
the SSP-RK3 method described in Gottlieb and Shu \cite{GoShu98},
with $m_x = 6000$ points and a small CFL number of 
$\nu = 0.1$.

\begin{figure}[!htbp]
\centering
\includegraphics[width=\textwidth]{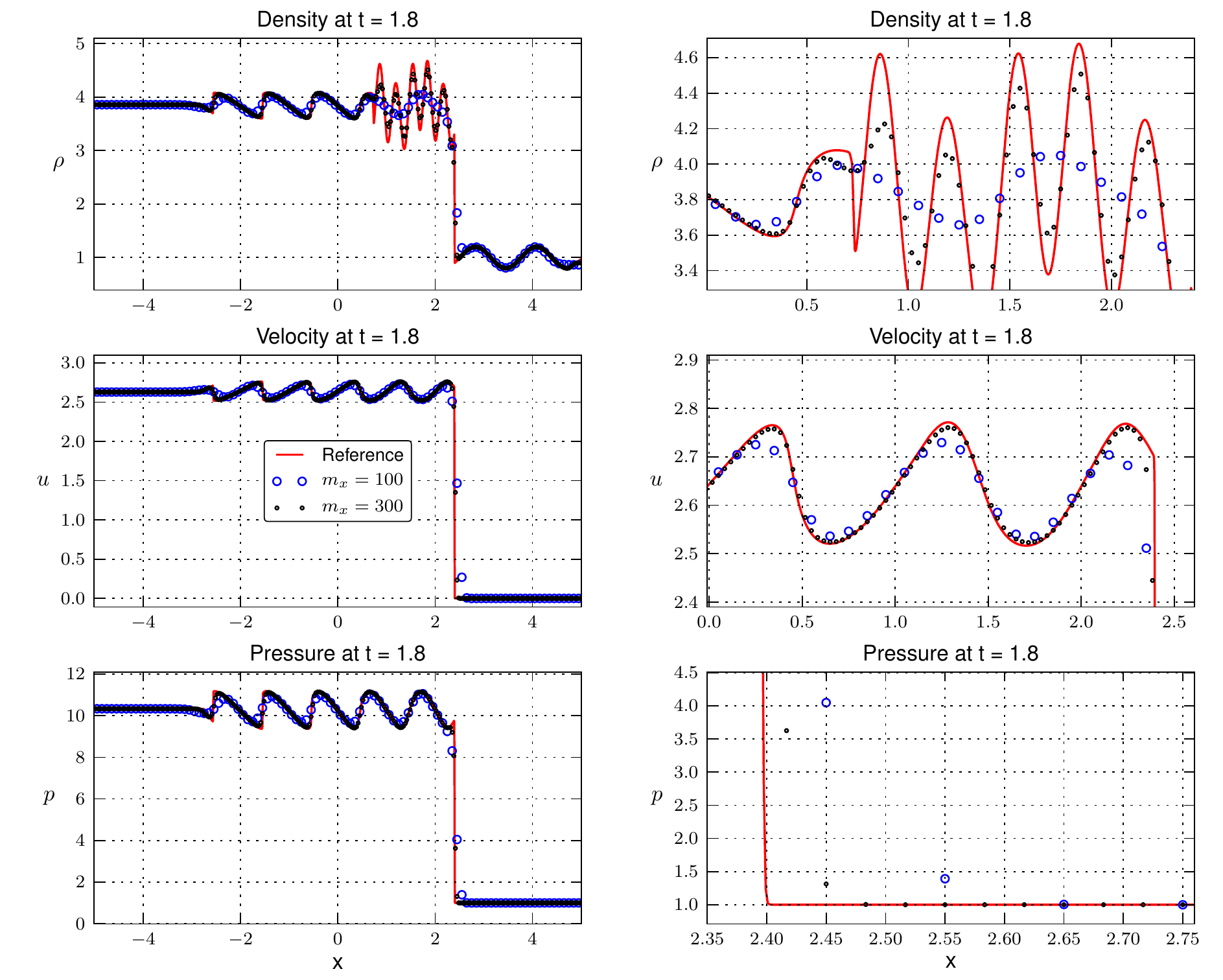}
\caption{
Shock-entropy: WENO solutions.
Shown here are WENO solutions to the shock-entropy interaction problem presented
in \S\ref{subsubsec:shock-entropy}.
We plot observable quantities from top to bottom: 
density $\rho$, 
velocity $u$,
and 
pressure $p$.
Left columns are the full solution, and right columns are zoomed in parts
of the same data points.
We present a coarse solution with $m_x = 100$ points and a finer solution
of $m_x = 300$ points.
}
\label{fig:euler-shock-entropy-weno}
\end{figure}

\begin{figure}[!htbp]
\begin{center}
\includegraphics[width=\textwidth]{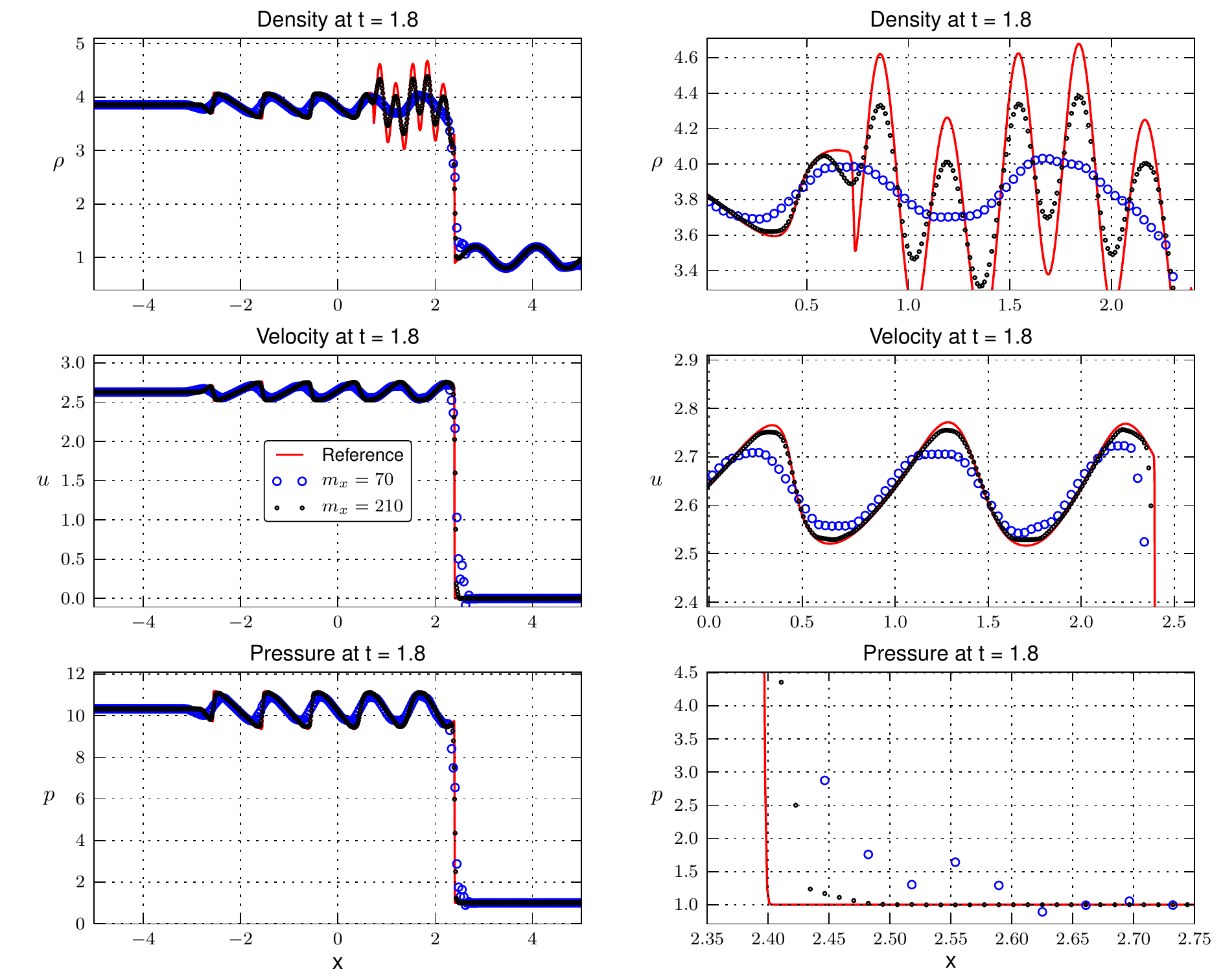}
\caption{Shock entropy.
Shown here are DG solutions to the shock-entropy interaction problem presented
in \S\ref{subsubsec:shock-entropy}.
We use a coarse mesh of $m_x = 70$ and a finer mesh of $m_x = 210$.
The reference solution is identical to the one used in
Figure 
\ref{fig:euler-shock-entropy-weno}
that is a WENO simulation with $6000$ points, and CFL number of $\nu = 0.1$
with the third-order SSP-RK3 time integrator.
Again, the DG schemes tend to be more diffusive on this problem when
compared to the WENO method, and the limiter used in this work tends to
clip the peaks for the coarse resolutions.
We repeat that we are plotting exactly four points per cell, and therefore,
the observable oscillations for the coarse resolution
are happening on a sub-cell level.
Although not shown, plots of cell averages only do not
visibly demonstrate these oscillations.
Results for other time integrators are found to be comparable.
\label{fig:euler-shock-entropy-dg}
}
\end{center}
\end{figure}

\section{Conclusions and future work}
\label{sec:conclusions}

We have presented a family of methods that generalize popular high-order
time integration methods for hyperbolic conservation laws.
The explicit multistage multiderivative (multiderivative Runge-Kutta) 
time integrators that are the subject
of this work provide the ability to
access to higher temporal derivatives for
an explicit (single-derivative) Runge-Kutta method, and they introduce
degrees of freedom by adding stages to 
high-order (single-stage) Taylor methods.
Numerous numerical examples were presented that included multiderivative schemes
for high-order discontinuous Galerkin
and WENO methods.
In order to implement the multiderivative technology, we 
leveraged recent work on Lax-Wendroff type time integrators
for the two aforementioned spatial discretizations investigated; each method
required a very different procedure for defining the higher derivatives.
Numerical results for the new multiderivative schemes
are promising: they are demonstrably comparable
to those obtained from popular high-order SSP integrators, 
they introduce greater portability to high-order Lax-Wendroff methods,
and they decrease the memory footprint for Runge-Kutta methods by introducing
higher time derivatives.


Future work will focus on extensions to higher dimensions,
as well as a mathematical exploration into the numerical properties of 
multiderivative schemes for PDEs.
In addition, an investigation into the optimization of these methods for modern
computer architectures such as graphics processing units (GPUs) should be conducted.
This will include implementation and timing comparison tests of these methods 
on GPUs.
We would like to explore developing multiderivative methods with
SSP properties, as well as low-storage `many'-stage variations of the 
two-derivative method presented in this work.
Additionally, we would like to investigate implicit and explicit multistage 
multiderivative methods for solving parabolic partial differential equations. \\




\noindent
\begin{acknowledgements}
This work has been supported in part by 
Air Force Office of Scientific Research 
grants 
FA9550-11-1-0281,  
FA9550-12-1-0343   
and
FA9550-12-1-0455,  
and by National Science Foundation grant number 
DMS-1115709.       
We would like to thank
Matthew F. Causley for discussing multiderivative methods with us, and Qi Tang for
useful discussions on the WENO method.
\end{acknowledgements}


\bibliographystyle{spmpsci}      

\bibliography{LxW-truncated}

\end{document}